\documentclass[reqno]{amsart}
\usepackage{amsmath}%
\usepackage{amsfonts}%
\usepackage{amssymb}%
\usepackage{graphicx}
\usepackage{mathrsfs}
\usepackage{hyperref}

\newtheorem{theorem}{Theorem}
\theoremstyle{plain}

\newtheorem{claim}[theorem]{Claim}

\newtheorem{corollary}[theorem]{Corollary}

\newtheorem{definition}[theorem]{Definition}

\newtheorem{fact}[theorem]{Fact}
\newtheorem{lemma}[theorem]{Lemma}

\newtheorem{proposition}[theorem]{Proposition}

\numberwithin{equation}{section}
\numberwithin{theorem}{section}

\def\F{\mathcal{F}}
\def\G{\mathcal{G}}
\def\h{\mathcal{H}}
\def\K{\mathcal{K}}

\def\pp{\mathcal{P}}
\def\Q{\mathcal{Q}}

\def\Y{\mathcal{Y}}
\def \a{\alpha}
\def \e{\epsilon}
\def \d{\text{deg}}
\def \r{\gamma}

\begin{document}
\title[Hamilton $\ell$-cycles]{Minimum codegree threshold for Hamilton $\ell$-cycles in $k$-uniform hypergraphs}
\thanks{The second author is partially supported NSA grant H98230-12-1-0283 and NSF grant DMS-1400073.}
\author{Jie Han}
\address[Jie Han and Yi Zhao]
{Department of Mathematics and Statistics,\newline
\indent Georgia State University, Atlanta, GA 30303}
\email[Jie Han]{jhan22@gsu.edu}
\author{Yi Zhao}
\email[Yi Zhao]{yzhao6@gsu.edu}%
\date{\today}
\begin{abstract}
For $1\le \ell<k/2$, we show that for sufficiently large $n$, every $k$-uniform hypergraph on $n$ vertices with minimum codegree at least $\frac n{2 (k-\ell)} $ contains a Hamilton $\ell$-cycle. This codegree condition is best possible and improves on work of H\`an and Schacht who proved an asymptotic result.
\end{abstract}

\maketitle

\section{Introduction}
A well-known result of Dirac \cite{Dirac} states that every graph $G$ on $n\ge 3$ vertices with minimum degree $\delta(G)\ge n/2$ contains a Hamilton cycle. 
In recent years, researchers have worked on extending this result to hypergraphs -- see recent surveys of \cite{KuOs14ICM, RR}.
Given $k\ge 2$, a $k$-uniform hypergraph (in short, \emph{$k$-graph}) consists of a vertex set
$V$ and an edge set $E\subseteq \binom{V}{k}$, where every edge is a $k$-element subset of $V$.
Given a $k$-graph $\h$ with a set $S$ of $d$ vertices (where $1 \le d \le k-1$) we define $\d_{\h} (S)$ to be the number of edges containing $S$ (the subscript $\h$ is omitted if it is clear from the context). The \emph{minimum $d$-degree $\delta _{d} (\h)$} of $\h$ is the minimum of $\d_{\h} (S)$ over all $d$-vertex sets $S$ in $\h$.  We refer to  $\delta _1 (\h)$ as the \emph{minimum vertex degree} and  $\delta _{k-1} (\h)$ the \emph{minimum codegree} of $\h$.
For $1\le \ell< k$, a $k$-graph is a called an \emph{$\ell$-cycle} if its vertices can be ordered cyclically such that each of its edges consists of $k$ consecutive vertices and every two consecutive edges (in the natural order of the edges) share exactly $\ell$ vertices. In $k$-graphs, a $(k-1)$-cycle is often called a \emph{tight} cycle while a $1$-cycle is often called a \emph{loose} cycle. We say that a $k$-graph contains a \emph{Hamilton $\ell$-cycle} if it contains an $\ell$-cycle as a spanning subhypergraph. 
Since a $k$-uniform $\ell$-cycle on $n$ vertices contains exactly $n/(k - \ell)$ edges, a necessary condition for a $k$-graph on $n$ vertices to contain a Hamilton $\ell$-cycle is that $k- \ell$ divides $n$.

Confirming a conjecture of Katona and Kierstead \cite{KK}, R\"odl, Ruci\'nski and Szemer\'edi \cite{RRS06, RRS08} showed that for any fixed $k$, every $k$-graph $\h$ on $n$ vertices with $\delta_{k-1}(\h)\ge n/2 + o(n)$ contains a tight Hamilton cycle. 
When $k-\ell$ divides both $k$ and $|V|$, a $(k-1)$-cycle on $V$ trivially contains an $\ell$-cycle on $V$. Thus the result in \cite{RRS08} implies that for all $1\le \ell<k$ such that $k- \ell$ divides $k$, every $k$-graph $\h$ on $n\in (k-\ell)\mathbb{N}$ vertices with $\delta_{k-1}(\h)\ge n/2 + o(n)$ contains a Hamilton $\ell$-cycle. 
It is not hard to see that these results are best possible up to the $o(n)$ term -- see Concluding Remarks for more discussion.
With long and involved arguments, R\"odl, Ruci\'nski and Szemer\'edi  \cite{RRS11} determined the minimum codegree threshold for tight Hamilton cycles in 3-graphs for sufficiently large $n$.
(Unless stated otherwise, we assume that $n$ is sufficiently large throughout the paper.)

Loose Hamilton cycles were first studied by K\"uhn and Osthus \cite{KO}, who proved that
every 3-graph on $n$ vertices with $\delta_2(\h)\ge n/4 + o(n)$ contains a loose Hamilton cycle. 
This was generalized to arbitrary $k$ and $\ell=1$ by Keevash, K\"uhn, Mycroft, and Osthus \cite{KKMO} and to arbitrary $k$ and arbitrary $\ell< k/2$ by H\`an and Schacht \cite{HS}.

\begin{theorem}\cite{HS}\label{thmHS}
Fix integers $k\ge 3$ and $1\le \ell <k/2$. Assume that $\gamma>0$ and $n\in (k-\ell)\mathbb N$ is sufficiently large. If $\h=(V, E)$ is a $k$-graph on $ n$ vertices such that $\delta_{k-1}(\h)\ge (\frac 1{2(k-\ell)}+\gamma)n$, then $\h$ contains a Hamilton $\ell$-cycle.
\end{theorem}

Later K\"uhn, Mycroft, and Osthus \cite{KMO} proved that whenever $k-\ell$ does not divide $k$, every $k$-graph on $n$ vertices with
$\delta_{k-1}(\h)\ge \frac{n}{\lceil \frac{k}{k-\ell} \rceil (k-\ell)} +o(n)$ contains a Hamilton $l$-cycle. 
This generalizes Theorem~\ref{thmHS} because $\lceil k/(k-\ell) \rceil =2 $ when $\ell< k/2$. 
R\"odl and Ruci\'nski \cite[Problem 2.9]{RR} asked for the exact minimum codegree threshold for Hamilton $\ell$-cycles in $k$-graphs. The $k=3$ and $\ell=1$ case was answered by Czygrinow and Molla \cite{CzMo} recently. In this paper we determine this threshold for all $k\ge 3$ and $\ell<k/2$.

\begin{theorem}[Main Result] \label{thmmain}
Fix integers $k\ge 3$ and $1\le \ell <k/2$. Assume that $n\in (k-\ell)\mathbb N$ is sufficiently large.
If $\h=(V, E)$ is a $k$-graph on $n$ vertices such that
\begin{equation}\label{eqdeg2}
\delta_{k-1}(\h) \ge \frac n{2 (k-\ell)},
\end{equation}
then $\h$ contains a Hamilton $\ell$-cycle.
\end{theorem}

The following simple construction \cite[Proposition 2.2]{KMO} shows that Theorem~\ref{thmmain} is best possible, 
and the aforementioned results in \cite{HS, KKMO, KMO, KO} are asymptotically best possible.
Let $\h_0=(V, E)$ be an $n$-vertex $k$-graph in which $V$ is partitioned into sets $A$ and $B$ such that $|A| = \left\lceil \frac{n}{\lceil \frac{k}{k- \ell} \rceil (k- \ell)} \right\rceil - 1$. The edge set $E$ consists of all $k$-sets that intersect $A$. It is easy to see that $\delta_{k-1}(\h_0)= |A|$.
However, an $\ell$-cycle on $n$ vertices has $n/(k-\ell)$ edges and every vertex on such a cycle lies in at most $\lceil \frac{k}{k-\ell} \rceil$ edges. Since $\lceil \frac{k}{k-\ell} \rceil |A|< n/(k-\ell)$, $\h_0$ contains no Hamilton $\ell$-cycle.

A related problem was studied by Bu\ss, H\`an, and Schacht \cite{BHS}, who proved that every 3-graph $\h$ on $n$ vertices with minimum vertex degree $\delta_1(\h)\ge (\frac 7{16}+o(1))\binom n2$ contains a loose Hamilton cycle. Recently we \cite{HZ1} improved this to an exact result.

\smallskip
Using the typical approach of obtaining exact results, our proof of  Theorem~\ref{thmmain} consists of an \emph{extremal case} and a \emph{nonextremal case}.

\begin{definition}
Let $\Delta>0$, a $k$-graph $\h$ on $n$ vertices is called $\Delta$-extremal if there is a set $B\subset V(\h)$, such that $|B| = \lfloor \frac{2(k-\ell)-1}{2(k-\ell)}n \rfloor $ and $e(B)\le \Delta n^k$.
\end{definition}

\begin{theorem}[Nonextremal Case]\label {lemNE}
For any integer $k\ge 3$, $1\le \ell <k/2$ and $0<\Delta<1$ there exists $\r>0$ such that the following holds. Suppose that $\h$ is a $k$-graph on $n$ vertices such that $n\in (k-\ell)\mathbb N$ is sufficiently large. If $\h$ is not $\Delta$-extremal and satisfies $\delta_{k-1}(\h)\ge (\frac 1{2(k-\ell)}-\r)n$, then $\h$ contains a Hamilton $\ell$-cycle.
\end{theorem}

\begin{theorem}[Extremal Case]\label {lemE}
For any integer $k\ge 3$, $1\le \ell <k/2$ there exists $\Delta>0$ such that the following holds. Suppose $\h$ is a $k$-graph on $n$ vertices such that $n\in (k-\ell)\mathbb N$ is sufficiently large. If $\h$ is $\Delta$-extremal and satisfies \eqref{eqdeg2}, then $\h$ contains a Hamilton $\ell$-cycle.
\end{theorem}

Theorem \ref{thmmain} follows from Theorem \ref{lemNE} and \ref{lemE} immediately by choosing $\Delta$ from Theorem \ref{lemE}.

Let us compare our proof with those in the aforementioned papers.
There is no extremal case in \cite{HS,KKMO,KMO,KO} because only asymptotic results were proved. Our Theorem~\ref{lemE} is new and more general than \cite[Theorem 3.1]{CzMo}. Following previous work \cite{HS, KMO, RRS06, RRS08, RRS11}, we prove Theorem~\ref{lemNE} by using the absorbing method initiated by R\"odl, Ruci\'nski and Szemer\'edi. More precisely, we find the desired Hamilton $\ell$-cycle by applying the Absorbing Lemma (Lemma~\ref{lemA}), the Reservoir Lemma (Lemma~\ref{lemR}), and the Path-cover Lemma (Lemma~\ref{lemP}). In fact, when $\ell< k/2$, the Absorbing Lemma and the Reservoir Lemma are not very difficult and already proven in \cite{HS} (in contrast, when $\ell> k/2$, the Absorbing Lemma in \cite{KMO} is more difficult to prove).
Thus the main step is to prove the Path-cover Lemma.
As shown in \cite{HS,KMO}, after the Regularity Lemma is applied, it suffices to prove that the cluster $k$-graph $\K$ can be tiled almost perfectly by the $k$-graph $\F_{k,\ell}$, whose vertex set consists of disjoint sets $A_1, \dots, A_{a-1}, B$ of size $k-1$, and whose edges are all the $k$-sets of the form $A_i\cup \{b\}$ for $i=1, \dots, a-1$ and all $b\in B$, where $a=\lceil \frac{k}{k-\ell} \rceil (k-\ell)$.
In this paper we reduce the problem to tile $\K$ with a much simpler $k$-graph $\Y_{k, 2\ell}$, which consists of two edges sharing $2\ell$ vertices. Because of the simple structure of $\Y_{k, 2\ell}$, we can easily find an almost perfect $\Y_{k, 2\ell}$-tiling unless $\K$ is in the extremal case (thus the original $k$-graph $\h$ is in the extremal case).
Interestingly $\Y_{3, 2}$-tiling was studied in the very first paper \cite{KO} on loose Hamilton cycles but as a separate problem. Our recent paper \cite{HZ1} indeed used $\Y_{3, 2}$-tiling as a tool to prove the corresponding path-cover lemma.
On the other hand, the authors of \cite{CzMo} used a different approach (without the Regularity Lemma) to prove the Path-tiling Lemma (though they did not state such lemma explicitly).

The rest of the paper is organized as follows: we prove Theorem~\ref{lemNE} in Section 2 and Theorem~\ref{lemE} in Section 3, and give concluding remarks in Section 4.

\smallskip
\textbf{Notation.} Given an integer $k\ge 0$, a $k$-set is a set with $k$ elements. For a set $X$, we denote by $\binom{X}{k}$ the family of all $k$-subsets of $X$. Given a $k$-graph $\h$ and a set $A \subseteq V(\h)$, we denote by $e_{\h}(A)$ the number of the edges of $\h$ in $A$. We often omit the subscript that represents the underlying hypergraph if it is clear from the context.
Given a $k$-graph $\h$ with two vertex sets $S, R$ such that $|S|<k$, we denote by $\deg_{\h} (S, R)$ the number of $(k-|S|)$-sets $T\subseteq R$ such that $S\cup T$ is an edge of $\h$ (in this case $T$ is called a \emph{neighbor} of $S$).
We define $\overline{\deg}_{\h}(S, R) = \binom{ |R\setminus S| }{k - |S|} - \deg(S, R)$ as the number of \emph{non-edges} on $S\cup R$ that contain $S$. When $R= V(\h)$ (and $\h$ is obvious), we simply write $\deg(S)$ and $\overline\deg(S)$. When $S= \{v\}$, we use $\deg(v, R)$ instead of $\deg(\{v\}, R)$.

A $k$-graph $\mathcal P$ is an \emph{$\ell$-path} if there is an ordering $(v_1, \dots, v_{t})$ of its vertices such that every edge consists of $k$ consecutive vertices and two consecutive edges intersect in exactly $\ell$ vertices. Note that this implies that $k-\ell$ divides $t-\ell$.
In this case we write $\mathcal{P} = v_1 \cdots v_t$ and call two $\ell$-sets $\{v_1, \dots, v_{\ell}\}$ and $\{v_{t-\ell+1}, \dots, v_t\}$ \emph{ends} of $\mathcal P$.

\section{Proof of Theorem \ref{lemNE}}
In this section we prove Theorem \ref{lemNE} by following the approach in \cite{HS}.

\subsection{Auxiliary lemmas and Proof of Theorem \ref{lemNE}}

We need \cite[Lemma 5]{HS} and \cite[Lemma 6]{HS} of H\`an and Schacht, in which only a linear codegree condition is needed.
Given a $k$-graph $\h$ with an $\ell$-path $\mathcal P$ and a vertex set $U\subseteq V(\h)\setminus V(\mathcal{P})$ with $|U|\in (k-\ell)\mathbb N$ , we say that $\mathcal P$ \emph{absorbs} $U$ if there exists an $\ell$-path $\mathcal Q$ of $\h$ with $V(\mathcal Q) = V(\mathcal P)\cup U$ such that $\mathcal P$ and $\mathcal Q$ have exactly the same ends.

\begin{lemma}[Absorbing lemma, \cite{HS}] \label{lemA}
For all integers $k\ge 3$ and $1\le \ell<k/2$ and every $\r_1>0$ there exist $\eta>0$ and an integer $n_0$ such that the following holds. Let $\h$ be a $k$-graph on $n\ge n_0$ vertices with $\delta_{k-1}(\h)\ge \r_1 n$. Then $\h$ contains an absorbing $\ell$-path $\mathcal P$ with $|V(\mathcal P)|\le \r_1^5 n$ that can absorb  any subset $U\subset V(\h) \setminus V(\mathcal P)$ of size $|U|\le \eta n$ and $|U|\in (k-\ell)\mathbb N$.
\end{lemma}

\begin{lemma}[Reservoir lemma, \cite{HS}]\label{lemR}
For all integers $k\ge 3$ and $1\le \ell<k/2$ and every $0<d, \r_2<1$ there exists an $n_0$ such that the following holds. Let $\h$ be a $k$-graph on $n>n_0$ vertices with $\delta_{k-1}(\h)\ge dn$, then there is a set $R$ of size at most $\r_2 n$ such that for all $(k-1)$-sets $S\in \binom V{k-1}$ we have $\deg(S, R)\ge d\r_2 n/2$.
\end{lemma}

The main step in our proof of Theorem \ref{lemNE} is the following lemma, which is stronger than \cite[Lemma 7]{HS}. We defer its proof to the next subsection.

\begin{lemma}[Path-cover lemma]\label{lemP}
For all integers $k\ge 3$, $1\le \ell <k/2$, and every $\r_3, \a>0$ there exist integers $p$ and $n_0$ such that the following holds. Let $\h$ be a $k$-graph on $n>n_0$ vertices with $\delta_{k-1}(\h)\ge (\frac1{2(k-\ell)} - \r_3)n$, then there is a family of at most $p$ vertex disjoint $\ell$-paths that together cover all but at most $\a n$ vertices of $\h$, or $\h$ is $14\r_3$-extremal.
\end{lemma}

We can now prove Theorem \ref{lemNE} in a similar way as in \cite{HS}.

\begin{proof}[Proof of Theorem \ref{lemNE}]
Given $k\ge 3$, $1\le \ell <k/2$ and $0<\Delta<1$, let $\r =\min\{ \frac{\Delta}{43}, \frac{1}{4k^2}\}$ and $n\in (k-\ell)\mathbb N$ be sufficiently large. Suppose that $\h = (V, E)$ is a $k$-graph on $n$ vertices with $\delta_{k-1}(\h)\ge (\frac1{2(k-\ell)} - \r)n$. Since $\frac1{2(k-\ell)} - \r>\r$, we can apply Lemma \ref{lemA} with $\r_1=\r$ and obtain $\eta>0$ and an absorbing path $\mathcal P_0$ with ends $S_0, T_0$ such that $|V(\mathcal P_0)|\le \r^5 n$ and $\mathcal{P}_0$ can absorb any $u$ vertices outside $\mathcal{P}_0$ if $u\le \eta n$ and $u\in (k-\ell)\mathbb N$.

Let $V_1 = (V\setminus V(\mathcal P_0)) \cup S_0 \cup T_0$ and $\h_1=\h[V_1]$.
Note that $|V(\mathcal P_0)|\le \r^5 n$ implies that $\delta_{k-1}(\h_1)\ge (\frac1{2(k-\ell)} - \r)n - \r^5 n\ge \frac1{2k}n$ as $\r<\frac1{4k^2}$ and $\ell\ge 1$.
We next apply Lemma \ref{lemR} with $d=\frac1{2k}$ and $\r_2=\min\{\eta/2, \r\}$ to $\h_1$ and get a reservoir $R\subset V_1$ with $|R|\le \r_2 |V(\h_1)|\le \r_2 n$ such that for any $(k-1)$-set $S\subset V_1$, we have
\begin{equation}\label{eq:Res}
\deg(S, R)\ge d\r_2 |V_1|/2 \ge d\r_2n/4.
\end{equation}
Let $V_2 = V\setminus (V(\mathcal P_0)\cup R)$, $n_2=|V_2|$, and $\h_2 = \h[V_2]$. Note that $|V(\mathcal P_0)\cup R|\le \r_1^5 n +\r_2 n \le 2\r n$, so
\[
\delta_{k-1}(\h_2)\ge \left(\frac1{2(k-\ell)} - \r \right)n - 2\r n \ge \left(\frac1{2(k-\ell)} - 3\r \right)n_2.
\]
Applying Lemma \ref{lemP} to $\h_2$ with $\r_3=3\r$ and $\a=\eta/2$, we obtain at most $p$ vertex disjoint $\ell$-paths that cover all but at most $\a n_2$ vertices of $\h_2$, unless $\h_2$ is $14\r_3$-extremal. In the latter case, there exists $B'\subseteq V_2$ such that $|B'| = \lfloor \frac{2k-2\ell-1}{2(k-\ell)}n_2 \rfloor$ and $e(B')\le 42\r n_2^k$. Then we add at most $n - n_2 \le 2\r n$ vertices from $V\setminus B'$ to $B'$ and obtain a vertex set $B\subseteq V(\h)$ such that $|B| = \lfloor \frac{2k-2\ell-1}{2(k-\ell)}n \rfloor$ and
\[
e(B)\le 42\r n_2^k + 2\r n \cdot \binom{n-1}{k-1} \le 42 \r n^k+ \r n^k \le \Delta n^k,
\]
which means that $\h$ is $\Delta$-extremal, a contradiction. In the former case, denote these $\ell$-paths by $\{\mathcal P_i\}_{i\in [p']}$ for some $p'\le p$, and their ends by $\{S_i, T_i\}_{i\in [p']}$. Note that both $S_i$ and $T_i$ are $\ell$-sets for $\ell<k/2$.
We arbitrarily pick disjoint $(k-2\ell-1)$-sets $X_0, X_1, \dots, X_{p'} \subset R\setminus (S_0\cup T_0)$ (note that $k- 2\ell -1\ge 0$).
Let $T_{p'+1} = T_0$. 
By \eqref{eq:Res}, as $d\r_2n/4 \ge k(p'+1)$, we may find $p'+1$ vertices $v_0, v_1, \dots, v_{p'}\in R$ such that $ S_i\cup T_{i+1} \cup X_i \cup \{v_i \}\in E(\h)$ for $0\le i\le p'$. We thus connect $\pp_0, \pp_1, \ldots, \pp_{p'}$ together and obtain an $\ell$-cycle $\mathcal C$. 
Note that 
\[
|V(\h)\setminus V(\mathcal C)|\le |R|+\a n_2\le \r_2 n + \a n\le \eta n
\] 
and $k-\ell$ divides $|V\setminus V(\mathcal C)|$ because $k-\ell$ divides both $n$ and $|V(\mathcal{C})|$. 
So we can use $\mathcal P_0$ to absorb all unused vertices in $R$ and uncovered vertices in $V_2$ thus obtaining a Hamilton $\ell$-cycle in $\h$.
\end{proof}

The rest of this section is devoted to the proof of Lemma \ref{lemP}.

\subsection{Proof of Lemma \ref{lemP}}

Following the approach in \cite{HS}, we use the Weak Regularity Lemma, which is a straightforward extension of Szemer\'edi's regularity lemma for graphs \cite{Sze}.

Let $\h = (V, E)$ be a $k$-graph and let $A_1, \dots, A_k$ be mutually disjoint non-empty subsets of $V$. We define $e(A_1, \dots, A_k)$ to be the number of 
\emph{crossing} edges, namely, those with one vertex in each $A_i$, $i\in [k]$, and the density of $\h$ with respect to ($A_1, \dots, A_k$) as
\[
d(A_1,\dots, A_k) = \frac{e(A_1, \dots, A_k)}{|A_1| \cdots|A_k|}.
\]
We say a $k$-tuple ($V_1, \dots, V_k$) of mutually disjoint subsets $V_1, \dots, V_k\subseteq V$ is \emph{$(\e, d)$-regular}, for $\e>0$ and $d\ge 0$, if
\[
|d(A_1, \dots, A_k) - d|\le \e
\]
for all $k$-tuples of subsets $A_i\subseteq V_i$, $i\in [k]$, satisfying $|A_i|\ge \e |V_i|$. We say ($V_1, \dots, V_k$) is \emph{$\e$-regular} if it is $(\e, d)$-regular for some $d\ge 0$. It is immediate from the definition that in an $(\e, d)$-regular $k$-tuple ($V_1, \dots, V_k$), if $V_i'\subset V_i$ has size $|V_i'| \ge c|V_i|$ for some $c\ge \e$, then ($V_1', \dots, V_k'$) is $(\e/c, d)$-regular.

\begin{theorem}[Weak Regularity Lemma]
\label{thmReg}
Given $t_0\ge 0$ and $\e>0$, there exist $T_0 = T_0(t_0, \e)$ and $n_0 = n_0(t_0,\e)$ so that for every $k$-graph $\h = (V, E)$ on $n>n_0$ vertices, there exists a partition $V = V_0 \cup V_1 \cup \cdots \cup V_t$ such that
\begin{enumerate}
\item[(i)] $t_0\le t\le T_0$,
\item[(ii)] $|V_1| = |V_2| = \cdots = |V_t|$ and $|V_0|\le \e n$,
\item[(iii)] for all but at most $\e \binom tk$ $k$-subsets $\{i_1,\dots, i_k\} \subset [t]$, the $k$-tuple $(V_{i_1}, \dots, V_{i_k})$ is $\e$-regular.
\end{enumerate}
\end{theorem}

The partition given in Theorem \ref{thmReg} is called an \emph{$\e$-regular partition} of $\h$. Given an $\e$-regular partition of $\h$ and $d\ge 0$, we refer to $ V_i, i\in [t]$ as \emph{clusters} and define the \emph{cluster hypergraph} $\K = \K(\e,d)$ with vertex set $[t]$ and $\{i_1,\dots,i_k\}\subset [t]$ is an edge if and only if $(V_{i_1}, \dots, V_{i_k})$ is $\e$-regular and $d(V_{i_1}, \dots, V_{i_k}) \ge d$.

We combine Theorem \ref{thmReg} and \cite[Proposition 16]{HS} into the following corollary, which shows that the cluster hypergraph almost inherits the minimum degree of the original hypergraph. 
Its proof is standard and similar as the one of \cite[Proposition 16]{HS} so we omit it.\footnote{Roughly speaking, the lower bound for $\deg_{\K}(S)$ contains $-d$ because when forming $\K$, we discard all $k$-tuple $(V_{i_1}, \dots, V_{i_k})$ of density less than $d$, contains $- \sqrt{\e}$ because at most $\e \binom tk$ $k$-tuple are not regular, and contains $-(k-1)$ because we discard all non-crossing edges of $\h$.}

\begin{corollary} \cite{HS}
\label{prop16old}
Given $c, \e, d>0$, integers $k\ge 3$ and $t_0$, there exist $T_0$ and $n_0$ such that the following holds. Let $\h$ be a $k$-graph on $n>n_0$ vertices with $\delta_{k-1}(\h)\ge c n$. Then $\h$ has an $\e$-regular partition $V_0 \cup V_1 \cup \cdots \cup V_t$ with $t_0\le t\le T_0$, and in the cluster hypergraph $\K = \K(\e,d)$, all but at most $\sqrt \e t^{k-1}$ $(k-1)$-subsets $S$ of $[t]$ satisfy $\deg_{\K}(S)\ge (c - d - \sqrt{\e})t - (k-1)$.
\end{corollary}

Let $\h$ be a $k$-partite $k$-graph with partition classes $V_1,\dots, V_k$. Given $1\le \ell <k/2$, we call an $\ell$-path $\pp$ with edges $\{e_1, \dots, e_q\}$ \emph{canonical} with respect to $(V_1,\dots, V_k)$ if
\[
e_i\cap e_{i+1}\subseteq \bigcup_{j\in [\ell]}V_j \quad\text{ or }\quad e_i\cap e_{i+1}\subseteq \bigcup_{j\in [2\ell]\setminus [\ell]}V_j
\]
for $i\in [q-1]$. 
When $j> 2\ell$, all $e_1\cap V_j, \dots, e_q\cap V_j$ are distinct and thus 
$|V(\pp) \cap V_j| = |(e_1\cup \cdots \cup e_q)\cap V_j| = q$. When $j\le 2\ell$, exactly one of $e_{i-1}\cap e_i$ and $e_i\cap e_{i+1}$ intersects $V_j$. Thus $|V(\mathcal P)\cap V_j |=\frac{q+1}2$ if $q$ is odd.

We need the following proposition from \cite{HS}.
\begin{proposition} \cite[Proposition 19]{HS}
\label{prop:lpath}
Suppose that $1\le \ell<k/2$ and $\h$ is a $k$-partite, $k$-graph with partition classes $V_1, \dots, V_k$ such that $|V_i|=m$ for all $i\in [k]$, and $|E(\h)|\ge d m^k$. Then there exists a canonical $\ell$-path in $\h$ with $t > \frac{d m}{2(k-\ell)}$ edges.
\end{proposition}

In \cite{HS} the authors used Proposition \ref{prop:lpath} to cover an $(\e,d)$-regular tuple $(V_1,\dots, V_k)$ of sizes $|V_1|= \cdots = |V_{k-1}| = (2k-2\ell-1)m$ and $|V_k|=(k-1)m$ with vertex disjoint $\ell$-paths.  Our next lemma shows that an $(\e,d)$-regular tuple $(V_1,\dots, V_k)$ of sizes $|V_1| = \cdots = |V_{2\ell}|= m$ and $|V_i| = 2m$ for $i> 2\ell$ can be covered with $\ell$-paths.

\begin{lemma}\label{lem:path}
Fix $k\ge 3$, $1\le \ell<k/2$ and $\e, d>0$ such that $d>2\e$. Let $m>\frac{k^2}{\e^2(d-\e)}$. Suppose $\mathcal V = (V_1, V_2,\dots, V_k)$ is an $(\e,d)$-regular $k$-tuple with
\begin{equation}
\label{eq:Vi}
|V_1| = \cdots = |V_{2\ell}|= m \quad \text{and} \quad |V_{2\ell+1}| = \cdots = |V_k|= 2m.
\end{equation}
Then there are at most $\frac{2k}{(d-\e)\e}$ vertex-disjoint $\ell$-paths that together cover all but at most $2k\e m$ vertices of $\mathcal V$.
\end{lemma}

\begin{proof}
We greedily find vertex-disjoint canonical $\ell$-paths of odd length by Proposition \ref{prop:lpath} in $\mathcal V$ until less than $\e m$ vertices are uncovered in $V_1$ as follows. Suppose that we have obtained $\ell$-paths $\mathcal P_1, \dots, \mathcal P_p$ for some $p\ge 0$. Let $q = \sum_{j=1}^p e(\mathcal{P}_j)$. Assume that  for all $j$, $\mathcal P_j$ is canonical with respect to $\mathcal V$ and $e(\mathcal{P}_j)$ is odd. Then $\bigcup_{j=1}^p \mathcal{P}_i$ contains $\frac{q+p}2$ vertices of $V_i$ for $i\in [2\ell]$ and $q$ vertices of $V_i$ for $i> 2\ell$.
For $i\in [k]$, let $U_i$ be the set of uncovered vertices of $V_i$ and assume that $|U_1|\ge \e m$.
Using \eqref{eq:Vi}, we derive that $|U_1| = \cdots = |U_{2\ell}|\ge \e m$ and
\begin{equation}
\label{eq:Ui}
|U_{2\ell+1}| = \cdots = |U_k|= 2|U_1| + p.
\end{equation}
We now consider a $k$-partite subhypergraph $\mathcal{V}'$ with arbitrary $|U_1|$ vertices in each $U_i$ for $i\in [k]$. By regularity, $\mathcal{V}'$ contains at least $(d-\e) |U_1|^k$ edges, so we can apply Proposition \ref{prop:lpath} and find an $\ell$-path of odd length at least $\frac{(d-\e) \e m}{2(k-\ell)}-1 \ge \frac{(d-\e) \e m}{2 k}$ (dismiss one edge if needed).
We continue this process until $|U_1|< \e m$. Let $\mathcal P_1, \dots, \mathcal P_p$ be the $\ell$-paths obtained in $\mathcal{V}$ after the iteration stops.
Since $|V_1\cap V(\mathcal{P}_j)|\ge \frac{(d-\e)\e m}{2k}$ for every $j$, we have
\[
p\le \frac{m}{\frac{(d-\e)\e m}{2k}}=\frac{2k}{(d-\e)\e}.
\]
Since $m> \frac{k^2}{\e^2(d-\e)}$, 
it follows that $p(k-2\ell)< \frac{2k^2}{(d-\e)\e} < 2\e m$.
By \eqref{eq:Ui}, the total number of uncovered vertices in $\mathcal V$ is
\begin{align*}
\sum_{i=1}^k |U_i| &= |U_1| 2\ell+(2|U_1|+p)(k-2\ell) = 2(k-\ell) |U_1| + p(k-2\ell) \\
&< 2(k-1)\e m+2\e m = 2k\e m. \qedhere
\end{align*}
\end{proof}

\smallskip
Given $k\ge 3$ and $0\le b<k$, let $\Y_{k,b}$ be a $k$-graph with two edges that share exactly $b$ vertices. In general, given two (hyper)graphs $\G$ and $\h$, a \emph{$\G$-tiling} is a sub(hyper)graph of $\h$ that consists of vertex-disjoint copies of $\G$. A $\G$-tiling is \emph{perfect} if it is a spanning sub(hyper)graph of $\h$. The following lemma is the main step in our proof of Lemma \ref{lemP} and
we prove it in the next subsection. Note that it generalizes \cite[Lemma 3.1]{CDN} of Czygrinow, DeBiasio, and Nagle.

\begin{lemma}[$\Y_{k,b}$-tiling Lemma] \label{lem:F}
Given integers $k\ge 3$, $1\le b<k$ and constants $\r,\beta >0$, there exist $0<\e'<\r \beta$ and an integer $n'$ such that the following holds. Suppose $\h$ is a $k$-graph on $n>n'$ vertices with $\deg(S)\ge (\tfrac1{2k-b} - \gamma)n$ for all but at most $\e' n^{k-1}$ sets $S\in \binom V{k-1}$,
then there is a $\Y_{k,b}$-tiling that covers all but at most $\beta n$ vertices of $\h$ unless $\h$ contains a vertex set $B$ such that $|B|=\lfloor \frac{2k - b-1}{2k-b}n\rfloor$ and $e(B)<6\r n^k$.
\end{lemma}

Now we are ready to prove Lemma \ref{lemP}.
\begin{proof}[Proof of Lemma \ref{lemP}]
Fix integers $k,\ell$, $0<\r_3,\a<1$. Let $\e', n'$ be the constants returned from Lemma \ref{lem:F} with $b=2\ell$, $\r=2\r_3$, and $\beta=\a/2$. Thus $\e'<\r \beta=\r_3 \a$. Let $T_0$ be the constant returned from Corollary \ref{prop16old} with $c=\frac 1{2(k- \ell)} - \r_3$, $\e=(\e')^2/16$, $d=\r_3/2$ and 
$t_0 > \max\{n', 4k/\r_3\}$. Furthermore, let $p= \frac{2T_0}{(d-2\e)\e}$. 

Let $n$ be sufficiently large and let $\h$ be a $k$-graph on $n$ vertices with $\delta_{k-1}(\h) \ge (\frac 1{2(k-\ell)} - \r_3) n$. Applying Corollary \ref{prop16old} with the constants chosen above, we obtain an $\e$-regular partition and a cluster hypergraph $\K = \K(\e, d)$ on $[t]$ such that for all but at most $\sqrt \e t^{k-1}$ $(k-1)$-sets $S\in \binom{[t]}{k-1}$,
\[
\deg_{\K}(S)\ge \left(\frac 1{2(k-\ell)} - \r_3 - d - \sqrt\e \right)t - (k-1) \ge \left(\frac 1{2(k-\ell)} - 2\r_3 \right)t,
\]
because $d=\r_3/2$, $\sqrt\e=\e'/4<\r_3/4$ and $k-1< \r_3 t_0/4\le \r_3 t/4$.
Let $m$ be the size of clusters, then $(1-\e)\frac nt\le m\le \frac nt$. Applying Lemma \ref{lem:F} with the constants chosen above, we derive that
either there is a $\Y_{k,2\ell}$-tiling $\mathscr Y$ of $\K$ which covers all but at most $\beta t$ vertices of $\K$ or there exists a set $B\subseteq V(\K)$, such that $|B| = \lfloor \frac{2k-2\ell-1}{2(k-\ell)}t \rfloor $ and $e_{\K}(B)\le 12\r_3 t^{k}$. In the latter case, let $B'\subseteq V(\h)$ be the union of the clusters in $B$. By regularity,
\[
e_{\h}(B')\le e_{\K}(B)\cdot m^k + \binom tk \cdot d \cdot m^k + \e \cdot \binom tk \cdot m^k + t \binom m2 \binom n{k-2},
\]
where the right-hand side bounds the number of edges from regular $k$-tuples with high density, edges from regular $k$-tuples with low density, edges from irregular $k$-tuples and edges that lie in at most $k-1$ clusters. 
Since $m\le \frac nt$, $\e < \r_3/ 16$, $d=\r_3/2$, and $t^{-1}< t_0^{-1}< \r_3/(4k)$, we obtain that
\begin{align*}
e_{\h}(B')\le 12\r_3 t^k \cdot \left( \frac nt \right)^k + \binom tk \frac{\r_3}{2}  \left( \frac nt \right)^k  + \frac{\r_3}{16} \binom tk \left( \frac nt \right)^k + t \binom{n/t}2 \binom n{k-2} <13\r_3 n^k.
\end{align*}
Note that $|B'|= \lfloor \frac{2k-2\ell-1}{2(k-\ell)}t \rfloor m \le \frac{2k-2\ell-1}{2(k-\ell)}t \cdot \frac nt= \frac{2k-2\ell-1}{2(k-\ell)} n$, and consequently $|B'|\le \lfloor \frac{2k-2\ell-1}{2(k-\ell)}n \rfloor$. On the other hand,
\begin{align*}
|B'| &= \left\lfloor \frac{2k-2\ell-1}{2(k-\ell)}t \right\rfloor m\ge \left( \frac{2k-2\ell-1}{2(k-\ell)}t-1 \right)(1-\e)\frac nt \ge \left( \frac{2k-2\ell-1}{2(k-\ell)}t - \e \frac{2k-2\ell-1}{2(k-\ell)}t - 1 \right)\frac nt\\
&\ge \left( \frac{2k-2\ell-1}{2(k-\ell)}t-\e t \right)\frac nt =\frac{2k-2\ell-1}{2(k-\ell)}n-\e n.
\end{align*}
By adding at most $\e n$ vertices from $V\setminus B'$ to $B'$, we get a set $B''\subseteq V(\h)$ of size exactly $\lfloor \frac{2k-2\ell-1}{2(k-\ell)}n \rfloor$, with $e(B'')\le e(B') + \e n \cdot n^{k-1}<14\r_3 n^k$. Hence $\h$ is $14\r_3$-extremal.

In the former case, let $m'=\lfloor m/2 \rfloor$. If $m$ is odd, we throw away one vertex from each cluster covered by $\mathscr Y$ (we do nothing if $m$ is even).
Thus, the union of the clusters covered by $\mathscr Y$ contains all but at most $\beta t m+|V_0| + t\le \a n/2 + 2\e n$ vertices of $\h$. 
We take the following procedure to each member $\Y'\in \mathscr Y$. 
Suppose that $\Y'$ has the vertex set $[2k-2\ell]$ with edges $\{1, \dots, k\}$ and $\{k-2\ell+1, \dots, 2k-2\ell\}$.
For $i\in [2k-2\ell]$, let $W_i$ denote the corresponding cluster in $\h$. We split each $W_i$, $i= k-2\ell+1, \dots, k$, into two disjoint sets $W_i^1$ and $W_i^2$ of equal size. Then each of the $k$-tuples $(W_{k-2\ell+1}^1, \dots, W_{k}^1, W_1,\dots, W_{k-2\ell})$ and $(W_{k-2\ell+1}^2, \dots, W_{k}^2, W_{k+1},\dots, W_{2k-2\ell})$ is
$(2\e ,d')$-regular for some $d'\ge d$ and of sizes $m', \dots, m'$, $2m', \dots, 2m'$. Applying Lemma \ref{lem:path} to these two $k$-tuples, we find a family of 
at most $\frac{2k}{(d'- 2\e)2\e}\le \frac{k}{(d - 2\e)\e}$ 
disjoint loose paths in each $k$-tuple covering all but at most $2k (2\e) m' \le 2k\e m$ vertices.
Since $|\mathscr Y|\le \frac t{2k-2\ell}$, we thus obtain a path-tiling that consists of at most 
$2 \frac t{2k-2\ell}\frac{k}{(d-2\e)\e} \le \frac{2T_0}{(d-2\e)\e} = p$ paths and covers all but at most
\[
2\cdot 2k\e m \cdot \frac{t}{2k-2\ell}+\a n/2+2\e n < 6\e n + \a n/2  < \a n
\]
vertices of $\h$, where we use $2k-2\ell>k$ and $\e=(\e')^2/16<(\r_3 \a)^2/16<\a/12$. This completes the proof.
\end{proof}

\subsection{Proof of Lemma \ref{lem:F}}
We first give an upper bound on the size of $k$-graphs containing no copy of $\Y_{k,b}$. In its proof, we use the concept of \emph{link (hyper)graph}: given a $k$-graph $\h$ with a set $S$ of at most $k-1$ vertices, the \emph{link graph} of $S$ is the $(k-|S|)$-graph with vertex set $V(\h)\setminus S$ and edge set $\{e\setminus S: e\in E(\h), S\subseteq e\}$. Throughout the rest of the paper, we frequently use the simple identity $\binom{m}{b} \binom{m-b}{k-b} = \binom{m}{k} \binom{k}{b}$, which holds for all integers $0\le b\le k\le m$.

\begin{fact}\label{fact:Y}
Let $0\le b< k$ and $m\ge 2k-b$. If $\h$ is a $k$-graph on $m$ vertices containing no copy of $\Y_{k,b}$, then $e(\h)< \binom m{k-1}$.
\end{fact}

\begin{proof}
Fix any $b$-set $S\subseteq V(\h)$ ($S=\emptyset$ if $b=0$) and consider its link graph $L_S$. Since $\h$ contains no copy of $\Y_{k,b}$, any two edges of $L_S$ intersect. Since $m\ge 2k-b$, the Erd\H{o}s--Ko--Rado Theorem \cite{EKR} implies that $|L_S|\le \binom{m-b-1}{k-b-1}$. Thus,
\begin{align*}
e(\h)&\le \frac1{\binom{k}{b}}\binom {m}{b} \cdot \binom{m-b-1}{k-b-1}= \frac1{\binom{k}{b}}\binom {m}{b} \binom{m-b}{k-b}\frac{k-b}{m-b}=\binom{m}{k} \frac{k-b}{m-b} \\
&= \binom{m}{k-1} \frac{k-b}{k} \, \frac{m- k+1}{m-b} < \binom{m}{k-1}. \qedhere
\end{align*}
\end{proof}

\smallskip
\begin{proof}[Proof of Lemma \ref{lem:F}]

Given $\r, \beta>0$, let $\e'=\frac{\r \beta^{k-1}}{(k-1)!}$ and let $n\in \mathbb{N}$ be sufficiently large. Let $\h$ be a $k$-graph on $n$ vertices that satisfies $\deg(S)\ge (\frac1{2k-b} - \r )n$ for all but at most $\e' n^{k-1}$ $(k-1)$-sets $S$. Let $\mathscr Y=\{\Y_1, \dots, \Y_m\}$ be a largest $\Y_{k,b}$-tiling in $\h$ (with respect to $m$) and write $V_i=V(\Y_i)$ for $i\in [m]$. Let $V'=\bigcup_{i\in [m]}V_i$ and $U=V(\h)\setminus V'$. Assume that $|U|>\beta n$ -- otherwise we are done.

Let $C$ be the set of vertices $v\in V'$ such that $\deg(v, U)\ge (2k-b)^2\binom{|U|}{k-2}$. We will show that $|C|\le \frac n{2k-b}$ and $C$ covers almost all the edges of $\h$, which implies that $\h[V\setminus C]$ is sparse and $\h$ is in the extremal case.
We first observe that every $\Y_i\in \mathscr Y$ contains at most one vertex in $C$. Suppose instead, two vertices $x,y\in V_i$ are both in $C$. Since $\deg(x, U)\ge (2k-b)^2\binom{|U|}{k-2}>\binom{|U|}{k-2}$, by Fact \ref{fact:Y}, there is a copy of $\Y_{k-1,b-1}$ in the link graph of $x$ on $U$, which gives rise to $\Y'$, a copy of $\Y_{k,b}$ on $\{x\}\cup U$.
Since the link graph of $y$ on $U\setminus V(\Y')$ has at least
\[
(2k-b)^2\binom{|U|}{k-2} - (2k-b-1)\binom{|U|}{k-2}>\binom{|U\setminus V(\Y')|}{k-2}
\]
edges, we can find another copy of $\Y_{k,b}$ on $\{y\}\cup (U\setminus V(\Y'))$ by Fact~\ref{fact:Y}. Replacing $\Y_i$ in $\mathscr Y$ with these two copies of $\Y_{k,b}$ creates a $\Y_{k,b}$-tiling larger than $\mathscr Y$, contradiction.
Consequently,
\begin{align}
\sum_{S\in \binom{U}{k-1}}\deg(S, V')&\le |C|\binom {|U|}{k-1} + |V'\setminus C| (2k-b)^2\binom{|U|}{k-2} \nonumber\\
&< |C|\binom {|U|}{k-1} + (2k-b)^2 n\binom{|U|}{k-2} \quad \text{because  } |V'\setminus C|< n \nonumber \\
&=\binom {|U|}{k-1}\left(|C| +\frac{(2k - b)^2 n (k-1)}{|U|-k+2} \right). \label{eq:C1}
\end{align}
Second, by Fact \ref{fact:Y}, $e(U)\le \binom{|U|}{k-1}$ since $\h[U]$ contains no copy of $\Y_{k,b}$, which implies
\begin{equation}\label{eq:C2}
\sum_{S\in \binom{U}{k-1}}\deg(S, U)\le k \binom{|U|}{k-1}.
\end{equation}
By the definition of $\e'$, we have
\[
\e' n^{k-1}=\frac{\r \beta^{k-1}}{(k-1)!} n^{k-1}< \frac{\r |U|^{k-1}}{(k-1)!}< 2\r \binom{|U|}{k-1}
\]
as $|U|$ is large enough. At last, by the degree condition, we have
\begin{align}
\sum_{S\in \binom{U}{k-1}}\deg(S)\ge \left(\binom{|U|}{k-1}-\e' n^{k-1}\right)\left(\frac1{2k - b} - \r \right)n
>(1-2\r)\binom{|U|}{k-1}\left(\frac1{2k - b} - \r \right)n, \label{eq:C3}
\end{align}
Since $\deg(S)=\deg(S, U)+\deg(S, V')$, we combine \eqref{eq:C1}, \eqref{eq:C2} and \eqref{eq:C3} and get
\[
|C|>(1-2\r)\left(\frac1{2k - b} - \r \right)n - k - \frac{(2k - b)^2 n (k-1)}{|U|-k+2}.
\]
Since $|U|>16 k^3/\r$, we get
\[
\frac{(2k - b)^2 n (k-1)}{|U|-k+2}<\frac{4k^3 n}{|U|/2} < \r n/2.
\]
Since $2\r^2 n>k$ and $2k-b\ge 4$, it follows that $|C|>\left(\frac1{2k-b} - 2\r \right)n$.

Let $I_C$ be the set of all $i\in [m]$ such that $V_i\cap C\neq \emptyset$. Since each $V_i$, $i\in I_C$, contains one vertex of $C$, we have
\begin{equation}\label{eq:IC}
    |I_C| = |C|\ge \left(\frac1{2k-b} - 2\r \right)n \ge m - 2\r n.
\end{equation}
Let $A = (\bigcup_{i\in I_C}V_i\setminus C) \cup U$.

\begin{claim}\label{clm:edge}
$\h[A]$ contains no copy of $\Y_{k,b}$, thus $e(A)< \binom n{k-1}$.
\end{claim}

\begin{proof}
The first half of the claim implies the second half by Fact \ref{fact:Y}. Suppose instead, $\h[A]$ contains a copy of $\Y_{k,b}$, denoted by $\Y_0$. Note that $V(\Y_0) \not\subseteq U$ because $\h[U]$ contains no copy of $\Y_{k,b}$.  Without loss of generality, suppose that $V_1, \dots, V_{j}$ contain the vertices of $\Y_0$ for some $j\le 2k - b$. For $i\in [j]$, let $c_i$ denote the unique vertex in $V_i\cap C$. We greedily construct vertex-disjoint copies of $\Y_{k,b}$ on $\{c_i\}\cup U$, $i\in [j]$ as follows. Suppose we have found $\Y'_1, \dots, \Y'_i$ (copies of $\Y_{k,b}$) for some $i<j$.
Let $U_0$ denote the set of the vertices of $U$ covered by $\Y_0, \Y'_1, \dots, \Y'_i$. Then $|U_0|\le (i+1) (2k-b-1)\le (2k-b)(2k-b-1)$. Since $\deg(c_{i+1}, U)\ge (2k-b)^2\binom{|U|}{k-2}$, the link graph of $c_{i+1}$ on $U\setminus U_0$ has at least
\[
(2k-b)^2\binom{|U|}{k-2} - |U_0| \binom{|U|}{k-2} >\binom{|U|}{k-2}
\]
edges. By Fact~\ref{fact:Y}, there is a copy of $\Y_{k,b}$ on $\{c_{i+1}\}\cup (U\setminus U_0)$. Let $\Y'_1, \dots, \Y'_j$ denote the copies of $\Y_{k,b}$  constructed in this way. Replacing $\Y_1,\dots, \Y_{j}$ in $\mathscr{Y}$ with $\Y_0, \Y'_1, \dots, \Y'_j$ gives a $\Y_{k,b}$-tiling larger than $\mathscr Y$, contradiction.
\end{proof}

Note that the edges not incident to $C$ are either contained in $A$ or intersect some $V_i$, $i\notin I_C$. By \eqref{eq:IC} and Claim \ref{clm:edge},
\begin{align*}
e(V\setminus C) &\le e(A) + (2k-b)\cdot 2\r n \binom {n-1}{k-1}  < \binom n{k-1}  + (4k - 2b)\r n \binom {n}{k-1} \\
& <4k\r n \binom {n}{k-1}<\frac{4k}{(k-1)!}\r n^k\le 6\r n^k,
\end{align*}
where the last inequality follows from $k\ge 3$. Since $|C|\le \frac n{2k - b}$, we can pick a set $B\subseteq V\setminus C$ of order $\lfloor \frac{2k - b-1}{2k-b}n\rfloor$ such that $e(B)<6\r n^k$.
\end{proof}

\section{The Extremal Theorem}
In this section we prove Theorem \ref{lemE}. Assume that $k\ge 3$, $1\le \ell < k/2$ and $0<\Delta\ll 1$.
Let $n\in (k - \ell) \mathbb{N}$ be sufficiently large. Let $\h$ be a $k$-graph on $V$ of $n$ vertices such that $\delta_{k-1}(\h) \ge \frac n{2 (k-\ell)}$. Furthermore, assume that $\h$ is $\Delta$-extremal, namely, there is a set $B\subseteq V(\h)$, such that $|B| =\lfloor \frac{(2k-2\ell-1)n }{2(k-\ell)}\rfloor$ and $e(B)\le \Delta n^k$. Let $A=V\setminus B$. Then $|A|= \lceil \frac{n}{2(k-\ell)} \rceil$.

The following is an outline of the proof. We denote by $A'$ and $B'$ the sets of the vertices of $\h$ that behave as typical vertices of $A$ and $B$, respectively. Let $V_0=V\setminus(A'\cup B')$. It is not hard to show that $A' \approx A$, $B' \approx B$, and thus $V_0 \approx \emptyset$.
In the ideal case when $V_0 = \emptyset$ and $|B'| = (2k - 2\ell -1) |A'| $, we assign a cyclic order to the vertices of $A'$, construct $|A'|$ copies of $\mathcal{Y}_{k, \ell}$ such that each copy contains one vertex of $A'$ and $2k - \ell -1$ vertices of $B'$, and any two consecutive copies of  $\mathcal{Y}_{k, \ell}$ share exactly $\ell$ vertices of $B'$. This gives rise to the desired Hamilton $\ell$-cycle of $\h$. In the general case, we first construct an $\ell$-path $\mathcal{Q}$ with ends $L_0$ and $L_1$ such that $V_0 \subseteq V(\Q)$ and $|B_1| = (2k - 2\ell - 1) |A_1| + \ell$, where $A_1 = A' \setminus V(\Q)$ and $B_1 = (B\setminus V(\Q)) \cup L_0 \cup L_1$. Next we complete the Hamilton $\ell$-cycle by constructing an $\ell$-path on $A_1 \cup B_1$ with ends $L_0$ and $L_1$.

For the convenience of later calculations, we let $\e_0 =2 k!e\Delta \ll 1$ and claim that $e(B)\le \e_0 \binom {|B|}k$.
Indeed, since $2(k-\ell)-1\ge k$, we have
\[
\frac1e\le \left(1-\frac{1}{2(k-\ell)} \right)^{2(k-\ell)-1}\le \left(1-\frac{1}{2(k-\ell)} \right)^{k}.
\]
Thus we get
\begin{equation}\label{eqB}
e(B)\le \frac{\e_0}{2k!e} n^k \le \e_0 \left(1-\frac{1}{2(k-\ell)} \right)^{k} \frac{n^k}{2k!} \le \e_0 \binom {|B|}k.
\end{equation}

In general, given two disjoint vertex sets $X$ and $Y$ and two integers $i, j\ge 0$, a set $S\subset X\cup Y$ is called an $X^i Y^j$-set if $|S\cap X|= i$ and $|S\cap Y| =j$. When $X, Y$ are two disjoint subsets of $V(\h)$ and $i+j=k$, we denote by $\h(X^i Y^j)$ the family of all edges of $\h$ that are $X^i Y^j$-sets, and let $e_{\h}(X^i Y^{j})= |\h(X^i Y^j)|$ (the subscript may be omitted if it is clear from the context).
We use $\overline{e}_{\h}(X^i Y^{k-i})$ to denote the number of non-edges among $X^i Y^{k-i}$-sets. Given a set $L\subseteq X\cup Y$ with $|L\cap X|=l_1\le i$ and $|L\cap Y|=l_2\le k-i$, we define $ \deg(L, X^i Y^{k-i})$ as the number of edges in $\h(X^i Y^{k-i})$ that contain $L$, and
$\overline \deg(L, X^i Y^{k-i})=\binom{|X|-l_1}{i-l_1} \binom{|Y|-l_2}{k-i-l_2} - \deg(L, X^i Y^{k-i})$. Our earlier notation $\deg(S, R)$ may be viewed as $\deg(S, S^{|S|} (R\setminus S)^{k- |S|})$.

\subsection{Classification of vertices}

Let $\e_1={\e_0}^{1/3} $ and $\e_2=2\e_1^2$. Assume that the partition $V(\h)=A\cup B$ satisfies that $|B| =\lfloor \frac{(2k-2\ell-1)n }{2(k-\ell)}\rfloor$ and \eqref{eqB}. In addition, assume that $e(B)$ is the smallest among all such partitions. We now define
\begin{align*}
&A':=\left\{ v\in V: \deg (v,B)\ge (1-\e_1)\binom{|B|}{k-1} \right\}, \\
&B':=\left\{ v\in V: \deg (v,B)\le \e_1\binom{|B|}{k-1} \right\}, \\
& V_0:=V\setminus(A'\cup B').
\end{align*}

\begin{claim}\label{clm:eB}
$A\cap B'\neq \emptyset$ implies that $B\subseteq B'$, and $B\cap A'\neq \emptyset$ implies that $A\subseteq A'$.
\end{claim}

\begin{proof}
First, assume that $A\cap B'\neq \emptyset$. Then there is some $u\in A$ such that $\deg(u,B)\le \e_1\binom{|B|}{k-1}$. If there exists some $v\in B\setminus B'$, namely, $\deg(v, B)>\e_1\binom{|B|}{k-1}$, then we can switch $u$ and $v$ and form a new partition $A''\cup B''$ such that $|B''|=|B|$ and $e(B'')<e(B)$, which contradicts the minimality of $e(B)$.

Second, assume that $B\cap A'\neq \emptyset$. Then some $u\in B$ satisfies that $\deg(u,B)\ge (1-\e_1)\binom{|B|}{k-1}$. Similarly, by the minimality of $e(B)$, we get that for any vertex $v\in A$, $\deg(v, B)\ge (1-\e_1)\binom{|B|}{k-1}$, which implies that $A\subseteq A'$.
\end{proof}

\begin{claim}\label{clm:size}
$\{|A\setminus A'|, |B\setminus  B'|, |A'\setminus  A|, |B'\setminus  B|\}\le \e_2|B|$ and $|V_0|\le 2\e_2|B|$.
\end{claim}

\begin{proof}
First assume that $|B\setminus B'|> \e_2|B|$. By the definition of $B'$, we get that
\[
e(B) > \frac 1k \e_1\binom {|B|}{k-1} \cdot \e_2 |B| > 2\e_0\binom {|B|}k,
\]
which contradicts \eqref{eqB}.

Second, assume that $|A\setminus A'|> \e_2|B|$. Then by the definition of $A'$, for any vertex $v\notin A'$, we have that $\overline \deg(v,B)> \e_1\binom{|B|}{k-1}$. So we get
\[
\overline e(AB^{k-1})> \e_2|B| \cdot \e_1 \binom{|B|}{k-1} = 2\e_0|B|  \binom{|B|}{k-1}.
\]
Together with \eqref{eqB}, this implies that
\begin{align*}
\sum_{S\in \binom B{k-1}}\overline\deg(S)&= k\overline{e}(B) + \overline{e}(AB^{k-1}) \\
                              &> k(1 - \e_0) \binom{|B|}k + 2\e_0|B| \binom{|B|}{k-1}  \\
                              &= ((1-\e_0)(|B|-k+1)+2\e_0|B|) \binom{|B|}{k-1} > |B|  \binom{|B|}{k-1}.
\end{align*}
where the last inequality holds because $n$ is large enough. By the pigeonhole principle, there exists a set $S\in \binom{B}{k-1}$, such that
$\overline\deg(S) > |B|=\lfloor \frac{(2k-2\ell-1)n }{2(k-\ell)}\rfloor$, contradicting \eqref{eqdeg2}.

Consequently,
\begin{align*}
&|A'\setminus A|=|A'\cap B|\le |B\setminus B'|\le \e_2|B|,   \\
&|B'\setminus B|=|A\cap B'|\le |A\setminus A'|\le \e_2|B|, \\
&|V_0|=|A\setminus A'|+|B\setminus B'|\le \e_2|B|+\e_2|B|=2\e_2 |B|. \qedhere
\end{align*}
\end{proof}

\subsection{Classification of $\ell$-sets in $B'$}
In order to construct our Hamilton $\ell$-cycle, we need to connect two $\ell$-paths. To make this possible, we want the ends of our $\ell$-paths to be $\ell$-sets in $B'$ that have high degree in $\h[A'B'^{k-1}]$. Formally, we call an $\ell$-set $L\subset V$ \emph{typical} if $\deg(L, B)\le \e_1 \binom{|B|}{k-\ell}$, otherwise \emph{atypical}. We prove several properties related to typical $\ell$-sets in this subsection.

\begin{claim}\label{fact:ltypical}
The number of atypical $\ell$-sets in $B$ is at most $\e_2\binom{|B|}{\ell}$.
\end{claim}

\begin{proof}
Let $m$ be the number of atypical $\ell$-sets in $B$. By \eqref{eqB}, we have
\[
\frac{m \e_1 \binom{|B|}{k-\ell}}{\binom{k}{\ell}} \le e(B)\le \e_0 \binom{|B|}{k},
\]
which gives that $m\le \frac{ \e_0 \binom{k}{l} \binom{|B|}{k}}{ \e_1 \binom{|B|}{k-\ell} } = \frac{\e_2}{2} \binom{|B| - k+ \ell}{\ell} <  \e_2\binom{|B|}{\ell}$.
\end{proof}

\begin{claim}\label{fact:typ_l}
Every typical $\ell$-set $L\subset B'$ satisfies $\overline\deg(L, A' B'^{k-1})\le 4k\e_1\binom{|B'|-\ell}{k-\ell-1} |A'|$.
\end{claim}

\begin{proof}
Fix a typical $\ell$-set $L\subset B'$ and consider the following sum,
\[
\sum_{L\subset D\subset B', |D|=k-1}\deg(D)=\sum_{L\subset D\subset B', |D|=k-1}(\deg(D, A') + \deg(D, B') + \deg(D, V_0)).
\]
By \eqref{eqdeg2}, the left hand side is at least $\binom{|B'|-\ell}{k-\ell-1}|A|$. On the other hand,
\[
\sum_{L\subset D\subset B', |D|=k-1}(\deg(D, B') + \deg(D, V_0)) \le (k-\ell) \deg(L, B') +\binom{|B'|-\ell}{k-\ell-1}|V_0|.
\]
Since $L$ is typical and $|B'\setminus B| \le \e_2 |B|$ (Claim \ref{clm:size}), we have
\begin{align*}
\deg(L, B') & \le \deg(L, B) + |B'\setminus B| \binom{|B'|-1}{k- \ell-1} \\
& \le \e_1 \binom{|B|}{k-\ell} + \e_2 |B| \binom{|B'|-1}{k- \ell-1}.
\end{align*}
Since $\e_2\ll \e_1$ and $| |B| - |B'| | \le \e_2 |B|$, it follows that
\[
(k-\ell) \deg(L, B') \le \e_1 |B| \binom{|B|-1}{k-\ell-1} + (k-\ell) \e_2 |B| \binom{|B'|-1}{k- \ell-1} \le 2\e_1 |B|\binom{|B'|- \ell}{k-\ell-1}.
\]
Putting these together and using Claim \ref{clm:size}, we obtain that
\begin{align*}
\sum_{L\subset D\subset B', |D|=k-1}\deg(D, A') &\ge \binom{|B'|-\ell}{k-\ell-1}\left(|A| - |V_0|\right) -2\e_1 |B|\binom{|B'|- \ell}{k-\ell-1} \\
&\ge \binom{|B'|-\ell}{k-\ell-1} \left(|A'| - 3\e_2 |B| - 2\e_1 |B|\right).
\end{align*}
Note that $\deg(L, A' B'^{k-1})= \sum_{L\subset D\subset B', |D|=k-1}\deg(D, A') $.
Since $|B|\le (2k- 2\ell -1)|A|\le (2k-2\ell) |A'|$, we finally derive that
\[
\deg(L, A' B'^{k-1})\ge  \binom{|B'|-\ell}{k-\ell-1} (1 - (2k- 2\ell)( 3\e_2 + 2\e_1))|A'| \ge
 (1 - 4k\e_1)\binom{|B'|-\ell}{k-\ell-1} |A'|.
\]
as desired.
\end{proof}

We next show that we can connect any two disjoint typical $\ell$-sets of $B'$ with an $\ell$-path of length two while avoiding any given set of $\frac n{4(k-\ell)}$ vertices of $V$.

\begin{claim}\label{clm:conn}
Given two disjoint typical $\ell$-sets $L_1,L_2$ in $B'$ and a vertex set $U \subseteq V$ with $|U|\le \frac n{4(k-\ell)}$, there exist a vertex $a\in A'\setminus U$ and a $(2k- 3\ell-1)$-set $C\subset B'\setminus U$ such that $L_1\cup L_2\cup \{a\}\cup C$ spans an $\ell$-path (of length two) ended at $L_1, L_2$.
\end{claim}

\begin{proof}
Fix two disjoint typical $\ell$-sets $L_1,L_2$ in $B'$. Using Claim~\ref{clm:size}, we obtain that $|U|\le \frac n{4(k-\ell)}\le \frac{|A|}{2}< \frac23 |A'|$ and
\[
\frac{n}{4(k-\ell)} \le \frac{|B|+1}{2(2k- 2\ell -1)} \le \frac{ (1+ 2\e_2) |B'|}{2k}< \frac{|B'|}{k}.
\]
Thus $|A'\setminus U| > \frac {|A'|}3$ and $|B'\setminus U|> \frac{k-1}{k}|B'|$. Consider a $(k-\ell)$-graph $\G$ on $(A'\cup B')\setminus U$ such that
an $A' B'^{k-\ell-1}$-set $T$ is an edge of $\G$ if and only if $T\cap U = \emptyset$ and
$T$ is a common neighbor of $L_1$ and $L_2$ in $\h$. By Claim \ref{fact:typ_l}, we have
\begin{align*}
\overline{e}({\G}) &\le
2\cdot 4k\e_1\binom{|B'|-\ell}{k-\ell-1} |A'| < 8k\e_1 \binom{\frac{k}{k-1}|B' \setminus U|}{k-\ell-1} \cdot 3 |A'\setminus U|\\
&\le 24 k\e_1 \left(\frac{k}{k-1}\right)^{k-1} \binom{|B'\setminus U|}{k-\ell-1} |A'\setminus U|.
\end{align*}
Consequently $e(\G)> \frac12 \binom{|B'\setminus U|}{k-\ell-1} |A'\setminus U|$. Hence there exists a vertex $a\in A'\setminus U$ such that $\deg_{\G}(a) > \frac12 \binom{|B'\setminus U|}{k-\ell-1} > \binom{|B'\setminus U|}{k-\ell-2}$. By Fact \ref{fact:Y}, the link graph of $a$ contains a copy of $\Y_{k-\ell-1, \ell-1}$ (two edges of the link graph sharing $\ell -1$ vertices). In other words, there exists a $(2k- 3\ell-1)$-set $C\subset B'\setminus U$ such that $C\cup \{a\}$ contains two edges of $\G$ sharing $\ell$ vertices. Together with $L_1, L_2$, this gives rise to the desired $\ell$-path (in $\h$) of length two ended at $L_1, L_2$.
\end{proof}

The following claim shows that we can always extend a typical $\ell$-set to an edge of $\h$ by adding one vertex from $A'$ and $k- \ell -1$ vertices from $B'$ such that 
every $\ell$-set of these $k - \ell - 1$ vertices is typical. This can be done even when at most $\frac n{4(k-\ell)}$ vertices of $V$ are not available.

\begin{claim}\label{clm:extend}
Given a typical $\ell$-set $L\subseteq B'$ and a set $U\subseteq V$ with $|U|\le \frac{n}{4(k-\ell)}$, there exists an $A'B'^{k-\ell-1}$-set $C\subset V\setminus U$ such that $L\cup C$ is an edge of $\h$ and every $\ell$-subset of $C\cap B'$ is typical.
\end{claim}

\begin{proof}
First, since $L$ is typical in $B'$, by Claim \ref{fact:typ_l}, $\overline\deg(L, A' B'^{k-1})\le 4k\e_1\binom{|B'|-\ell}{k-\ell-1} |A'|$. Second, note that a vertex in $A'$ is contained in $\binom{|B'|}{k-\ell -1}$ $A' B'^{k-\ell-1}$-sets, while a vertex in $B'$ is contained in $|A'| \binom{|B'|-1}{k-\ell-2}$ $A' B'^{k-\ell-1}$-sets. It is easy to see that $|A'| \binom{|B'|-1}{k-\ell-2} < \binom{|B'|}{k-\ell -1}$  (as $|A'|\approx \frac{n}{2k-2\ell}$ and $|B'| \approx \frac{ 2k - 2\ell -1}{2k- 2\ell} n$). We thus derive that
at most
\[|U|\binom{|B'|}{k-\ell-1}\le \frac{n}{4(k-\ell)}\binom{|B'|}{k-\ell-1}
 \]
$A'B'^{k-\ell-1}$-sets intersect $U$. Finally, by Claim \ref{fact:ltypical}, the number of atypical $\ell$-sets in $B$ is at most $\e_2 \binom{|B|}{\ell}$. Using Claim~\ref{clm:size},
we derive that the number of atypical $\ell$-sets in $B'$ is at most
\[
\e_2 \binom{|B|}{\ell} + |B'\setminus B| \binom{|B'|-1}{\ell-1}\le 2\e_2 \binom{|B'|}{\ell} + \e_2 |B|\binom{|B'|-1}{\ell-1}<  3\ell \e_2 \binom{|B'|}{\ell}.
\]
Hence at most $3\ell \e_2 \binom{|B'|}{\ell} |A'| \binom{|B'|-\ell}{k-2\ell-1}$ $A'B'^{k-\ell-1}$-sets contain an atypical $\ell$-set. In summary, at most
\[
4k\e_1\binom{|B'| - \ell}{k-\ell-1} |A'| +\frac{n}{4(k-\ell)}\binom{|B'|}{k-\ell-1} + 3\ell \e_2 \binom{|B'|}{\ell}\binom{|B'|-\ell}{k-2\ell-1}|A'|
\]
$A'B'^{k-\ell-1}$-sets fail some of the desired properties. Since $\e_1, \e_2\ll 1$ and $|A'|\approx \frac{n}{2(k-\ell)}$, the desired $A'B'^{k-\ell-1}$-set always exists.
\end{proof}

\subsection{Building a short path $\mathcal Q$}
First, by the definition of $B$, for any vertex $b\in B'$, we have
\begin{align}
\label{eq:analeft}
\deg \left(b, B' \right)&\le \deg \left(b,B \right)+|B'\setminus B|\binom{|B'|-1}{k-2}  \nonumber    \\
                                & \le \e_1 \binom {|B|}{k-1} + \e_2 |B| \binom{|B'|-1}{k-2}< 2\e_1 \binom {|B|}{k-1}.
\end{align}

The following claim is the only place where we used the exact codegree condition \eqref{eqdeg2}.
\begin{claim}\label{clm:Bpath}
Suppose that $|A\cap B'|=q>0$. Then there exists a family $\mathcal P_1$ of $2q$ vertex-disjoint edges in $B'$, each of which contains two disjoint typical $\ell$-sets.
\end{claim}

\begin{proof}
Let $|A\cap B'|=q>0$. Since $A\cap B'\neq \emptyset$, by Claim \ref{clm:eB}, we have $B\subseteq B'$, and consequently $|B'|=\lfloor \frac{2k-2\ell-1 }{2(k-\ell)}n\rfloor+q$. By Claim \ref{clm:size}, we have $ q \le |A\setminus A'|\le \e_2|B|$.

Let $\mathcal B$ denote the family of the edges in $B'$ that contain two disjoint typical $\ell$-sets. We derive a lower bound for $|\mathcal{B}|$ as follows. We first pick a $(k-1)$-subset of $B$ (recall that $B\subseteq B'$) that contains no atypical $\ell$-subset. Since $2\ell\le k-1$, such a $(k-1)$-set contains two disjoint typical $\ell$-sets. By Claim \ref{fact:ltypical}, there are at most $\e_2 \binom{|B|}{\ell}$ atypical $\ell$-sets in $B\cap B'=B$ and in turn, there are at most $\e_2 \binom{|B|}{\ell} \binom{|B|-\ell }{k- \ell -1}$ $(k-1)$-subsets of $B$ that contain an atypical $\ell$-subset. Thus there are at least
\[
\binom{|B|}{k-1} - \e_2\binom{|B|}{\ell}\binom{|B|-\ell}{k-\ell-1} =  \left(1 - \binom{k-1}{\ell}\e_2 \right) \binom{|B|}{k-1}
\]
$(k-1)$-subsets of $B$ that contain no atypical $\ell$-subset. After picking such a $(k-1)$-set $S\subset B$, we find
a neighbor of $S$ by the codegree condition.  Since $|B'|=\lfloor \frac{2k-2\ell-1}{2(k-\ell)}n\rfloor+q$, by \eqref{eqdeg2}, we have $\deg \left(S, B'\right) \ge q$. We thus derive that
\[
|\mathcal{B}| \ge \left(1 - \binom{k-1}{\ell}\e_2 \right) \binom{|B|}{k-1} \frac{q}{k},
\]
in which we divide by $k$ because every edge of $\mathcal{B}$ is counted at most $k$ times.

We claim that $\mathcal B$ contains $2q$ disjoint edges. Suppose instead, a maximum matching in $\mathcal B$ has $i<2q$ edges. 
By \eqref{eq:analeft}, at most $2q k \cdot 2\e_1 \binom {|B|}{k-1}$ edges of $B'$ intersect the $i$ edges in the matching. Hence, the number of edges of $\mathcal{B}$ that are disjoint from these $i$ edges is at least
\begin{align*}
\frac{q}{k} \left(1 - \binom{k-1}{\ell}\e_2 \right) \binom{|B|}{k-1}-   4k\e_1 q \binom {|B|}{k-1}\ge \left(\frac{1}{k} - (4k+1)\e_1 \right) q\binom {|B|}{k-1} > 0,
\end{align*}
as $\e_2 \ll \e_1\ll 1$. We may thus obtain a matching of size $i+1$, a contradiction.
\end{proof}

\begin{claim}\label{clm:path}
There exists a non-empty $\ell$-path $\mathcal Q$ in $\h$ with the following properties:
\begin{itemize}
\item $V_0\subseteq V(\mathcal Q)$,
\item $|V(\mathcal Q)|\le 10k \e_2 |B|$,
\item the two ends $L_0, L_1$ of $\mathcal Q$ are typical $\ell$-sets in $B'$,
\item $|B_1|= (2k-2\ell-1)|A_1| + \ell$, where $A_1=A'\setminus V(\mathcal Q)$ and $B_1= (B'\setminus V(\mathcal Q))\cup L_0 \cup L_1$.
\end{itemize}
\end{claim}

\begin{proof}
We split into two cases here.

\smallskip
\noindent\textbf{Case 1.} $A\cap B'\neq \emptyset$.
\smallskip

By Claim \ref{clm:eB}, $A\cap B'\neq \emptyset$ implies that $B\subseteq B'$.
Let $q=|A\cap B'|$. We first apply Claim \ref{clm:Bpath} and find a family $\mathcal P_1$ of vertex-disjoint $2q$ edges in $B'$. Next we associate each vertex of $V_0$ with $2k - \ell -1$ vertices of $B$ (so in $B'$) forming an
$\ell$-path of length two such that these $|V_0|$ paths are pairwise vertex-disjoint, and also vertex-disjoint from the paths in $\mathcal P_1$, and all these paths have typical ends. To see it, let $V_0=\{x_1,\dots, x_{|V_0|}\}$.
Suppose that we have found such $\ell$-paths for $x_1, \dots, x_{i-1}$ with $i \le |V_0|$.
Since $B\subseteq B'$, it follows that $A\setminus A'=(A\cap B')\cup V_0$.
Hence $|V_0|+q=|A\setminus A'|\le \e_2 |B|$ by Claim \ref{clm:size}. Therefore
\[
(2k-\ell-1)(i -1) + | V(\pp_1) | < 2k|V_0|+2kq\le 2k \e_2|B|
\]
and consequently at most $2k\e_2 |B| \binom{|B|-1}{k-2} < 2k^2\e_2 \binom{|B|}{k-1}$ $(k-1)$-sets of $B$ intersect the existing paths (including $\mathcal P_1$).
By the definition of $V_0$, $\deg(x_{i}, B)> \e_1\binom{|B|}{k-1}$. Let $\G_{x_{i}}$ be the $(k-1)$-graph on $B$ such that $e\in \G_{x_{i}}$ if
\begin{itemize}
\item $\{x_{i}\}\cup e \in E(\h)$,
\item $e$ does not contain any vertex from the existing paths,
\item $e$ does not contain any atypical $\ell$-set.
\end{itemize}
By Claim \ref{fact:ltypical},
the number of $(k-1)$-sets in $B$ containing at least one atypical $\ell$-set is at most $\e_2\binom{|B|}{\ell} \binom{|B|-\ell}{k-\ell-1}=\e_2\binom{k-1}{\ell}\binom{|B|}{k-1}$. Thus, we have
\[
e(\G_{x_i})\ge \e_1\binom{|B|}{k-1} - 2k^2\e_2 \binom{|B|}{k-1} - \e_2\binom{k-1}{\ell}\binom{|B|}{k-1} >\frac{\e_1}{2} \binom{|B|}{k-1}>\binom{|B|}{k-2},
\]
because $\e_2 \ll \e_1$ and $|B|$ is sufficiently large.
By Fact \ref{fact:Y}, $\G_{x_i}$ contains a copy of $\Y_{k-1, \ell-1}$, which gives the desired $\ell$-path of length two containing $x_i$.

Denote by $\mathcal P_2$ the family of $\ell$-paths we obtained so far. Now we need to connect paths of $\mathcal P_2$ together to a single $\ell$-path. For this purpose, we apply Claim \ref{clm:conn} repeatedly to connect the ends of two $\ell$-paths while avoiding previously used vertices. This is possible because $|V(\mathcal P_2)|= (2k-\ell)|V_0|+2kq$ and $(2k-3\ell)(|V_0|+2q-1)$ vertices are needed to connect all the paths in $\pp_2$ -- the set $U$ (when we apply Claim \ref{clm:conn}) thus satisfies
\[
|U| \le (4k- 4\ell) |V_0| + (6k- 6\ell) q - 2k + 3\ell \le 6(k-\ell) \e_2 |B| -2k + 3\ell.
\]
Let $\mathcal P$ denote the resulting $\ell$-path.  We have $|V(\mathcal P)\cap A'|= |V_0|+2q-1$ and
\begin{align*}
|V(\mathcal P)\cap B'| &= k\cdot 2q+ (2k-\ell-1)|V_0| +(2k-3\ell-1)(|V_0|+2q-1)  \\
&=2(2k- 2\ell-1) |V_0| + 2(3k - 3\ell -1)q - (2k - 3\ell -1).
\end{align*}
Let $s= (2k-2\ell-1)|A'\setminus V(\mathcal P)| - |B'\setminus V(\mathcal P)|$. We have
\begin{align*}
s &= (2k - 2\ell - 1)(|A'| - |V_0| - 2q + 1) - |B'| +  2(2k- 2\ell-1) |V_0| + 2(3k - 3\ell -1)q - (2k - 3\ell -1) \\
 &= (2k - 2\ell -1) |A'| - |B'| + (2k- 2\ell-1) |V_0| + (2k - 2\ell)q + \ell.
\end{align*}
Since $|A'| + |B'| + |V_0| = n$, we have
\begin{equation}\label{eq:s}
    s= (2k - 2\ell) (|A'| + |V_0| + q) - n + \ell.
\end{equation}
Note that $|A'| + |V_0| + q = |A|$ and
\begin{equation}\label{eq:An}
    (2k - 2\ell) |A| - n =
    \begin{cases} 0, & \mbox{if } \frac{n}{k-\ell} \mbox{ is even} \\ k- \ell, & \mbox{if } \frac{n}{k-\ell} \mbox{ is odd}. \end{cases}
\end{equation}
Thus $s= \ell$ or $s=k$. If $s=  k$, then we extend $\pp$ to an $\ell$-path $\mathcal Q$ by applying Claim \ref{clm:extend}, otherwise let $\mathcal Q=\mathcal P$. Then
\[
|V(\mathcal Q)| \le |V(\mathcal P)|+ (k- \ell) \le 6k \e_2 |B|,
\]
and $\mathcal{Q}$ has two typical ends $L_0, L_1\subset B'$. We claim that
\begin{equation}\label{eq:ABQ}
(2k-2\ell-1)|A'\setminus V(\mathcal Q)| - |B'\setminus V(\mathcal Q)| = \ell.
\end{equation}
Indeed, when $s= \ell$, this is obvious; when $s=k$, $V(\Q)\setminus V(\pp)$ contains one vertex of $A'$ and $k-\ell-1$ vertices of $B'$ and thus
\[
(2k-2\ell-1)|A'\setminus V(\mathcal Q)| - |B'\setminus V(\mathcal Q)|=s - (2k-2\ell-1) + (k-\ell-1) =\ell.
\]
Let $A_1=A'\setminus V(\mathcal Q)$ and $B_1= (B'\setminus V(\mathcal Q))\cup L_0\cup L_1$. We derive that $|B_1|=(2k-2\ell-1)|A_1| + \ell$ from \eqref{eq:ABQ}.

\smallskip
\noindent\textbf{Case 2.} $A\cap B'= \emptyset$.
\smallskip

Note that $A\cap B'= \emptyset$ means that $B'\subseteq B$. Then we have
\begin{equation}\label{eq:A'B'}
|A'| + |V_0| = |V\setminus B'| = |A| + |B\setminus B'|.
\end{equation}

If $V_0\neq \emptyset$, we handle this case similarly as in Case 1 except that we do not need to construct $\mathcal P_1$.
By Claim \ref{clm:size}, $|B\setminus B'|\le \e_2|B|$ and thus for any vertex $x\in V_0$,
\begin{align}
\deg(x, B')&\ge \deg(x,B)-|B\setminus B'|\cdot \binom{|B|-1}{k-2} \nonumber \\
&\ge \e_1\binom{|B|}{k-1} - (k-1)\e_2\binom{|B|}{k-1}>\frac{\e_1}{2}\binom{|B'|}{k-1}. \label{eq:medium}
\end{align}
As in Case 1, we let $V_0=\{x_1,\dots, x_{|V_0|}\}$ and cover them with vertex-disjoint $\ell$-paths of length two. Indeed, for each $i  \le |V_0|$, we construct $\G_x$ as before and show that $e(\G_{x_i})\ge \frac{\e_1}{4} \binom{|B'|}{k-1}$. We then apply Fact \ref{fact:Y} to $\G_{x_i}$ obtaining a copy of $\Y_{k-1, \ell-1}$, which gives an $\ell$-path of length two containing $x_i$.
As in Case 1, we connect these paths to a single $\ell$-path $\mathcal P$ by applying Claim \ref{clm:conn} repeatedly. Then $|V(\mathcal P)|=(2k-\ell)|V_0|+(2k-3\ell)(|V_0|-1)$. Define $s$ as in Case~1. Thus \eqref{eq:s} holds with $q=0$. Applying \eqref{eq:A'B'} and \eqref{eq:An}, we derive that
\begin{equation}\label{eq:s2}
s = 2(k-\ell) (|A|+ |B\setminus B'|) - n +\ell =
\begin{cases} \ell + 2(k-\ell)|B\setminus B'|, & \mbox{if } \frac{n}{k-\ell} \mbox{ is even} \\
k+ 2(k-\ell)|B\setminus B'|, & \mbox{if } \frac{n}{k-\ell} \mbox{ is odd}, \end{cases}
\end{equation}
which implies that $s\equiv \ell \mod (k-\ell)$. We extend $\mathcal P$ to an $\ell$-path $\mathcal Q$ by applying Claim \ref{clm:extend} $\frac{s-\ell}{k-\ell}$ times. Then
\[
|V(\mathcal Q)|=|V(\mathcal P)|+s-\ell \le (4k - 4\ell)|V_0| - 2k+ 3\ell + k- \ell + 2(k-\ell)|B\setminus B'|\le 10k \e_2 |B|
\]
by Claim \ref{clm:size}. Note that $\mathcal{Q}$ has two typical ends
$L_0, L_1\subset B'$. Since $V(\Q)\setminus V(\pp)$ contains $\frac{s-\ell}{k-\ell}$ vertices of $A'$ and $\frac{s-\ell}{k-\ell}(k-\ell-1)$ vertices of $B'$, we have
\[
(2k-2\ell-1)|A'\setminus V(\mathcal Q)| - |B'\setminus V(\mathcal Q)|=s - \frac{s-\ell}{k-\ell}(2k-2\ell-1) + \frac{s-\ell}{k-\ell}(k-\ell-1) =\ell.
\]
We define $A_1$ and $B_1$ in the same way and similarly we have $|B_1|=(2k-2\ell-1)|A_1| + \ell$.

When $V_0=\emptyset$, we pick an arbitrary vertex $v\in A'$ and form an $\ell$-path $\mathcal{P}$ of length two with typical ends such that $v$ is in the intersection of the two edges. This is possible by the definition of $A'$. Define $s$ as in Case~1. It is easy to see that \eqref{eq:s2} still holds. We then extend $\mathcal P$ to $\mathcal Q$ by applying Claim \ref{clm:extend} $\frac{s-\ell}{k-\ell}$ times. Then
$|V(\mathcal{Q})|= 2k- \ell + s- \ell\le 2k \e_2 |B|$ because of \eqref{eq:s2}. The rest is the same as in the previous case.
\end{proof}

\begin{claim}
The $A_1, B_1$ and $L_0, L_1$ defined in Claim~\ref{clm:path} satisfy the following properties:
\begin{enumerate}
\item $|B_1|\ge (1- \e_1) |B|$,
\item for any vertex $v\in A_1$, $\overline{\deg}(v, B_1)<3 \e_1 \binom{|B_1|}{k-1}$,
\item for any vertex $v\in B_1$, $\overline{\deg}(v, A_1 B_1^{k-1})\le 3k\e_1 \binom{|B_1|}{k-1}$,
\item $\overline{\deg}(L_0, A_1B_1^{k-1}) \le 5k\e_1\binom{|B_1|}{k-\ell}$, $\overline{\deg}(L_1, A_1B_1^{k-1}) \le 5k\e_1\binom{|B_1|}{k-\ell}$.
\end{enumerate}
\end{claim}

\begin{proof}
Part (1): By Claim \ref{clm:size}, we have $|B_1\setminus B|\le |B'\setminus B|\le \e_2 |B|$. Furthermore,
\[
|B_1|\ge |B'| - |V(\mathcal Q)| \ge |B| - \e_2 |B|- 10k \e_2 |B| \ge (1-\e_1)|B|.
\]

Part (2): For a vertex $v\in A_1$, since $\overline{\deg}(v, B)\le \e_1 \binom{|B|}{k-1}$, we have
\begin{align*}
\overline{\deg}(v, B_1) &\le \overline{\deg}(v, B) + |B_1\setminus B| \binom{|B_1|-1}{k-2} \\
&\le \e_1 \binom{|B|}{k-1} + \e_2 |B| \binom{|B_1|-1}{k-2} \\
&< \e_1 \binom{|B|}{k-1} + \e_1 \binom{|B_1|}{k-1} <3 \e_1 \binom{|B_1|}{k-1},
\end{align*}
where the last inequality follows from Part~(1).

Part~(3): Consider the sum $\sum\deg(S\cup \{v\})$ taken over all $S\in \binom{B'\setminus \{v\}}{k-2}$. Since $\delta_{k-1}(\h)\ge |A|$, we have $\sum\deg(S\cup \{v\})\ge \binom{|B'|-1}{k-2}|A|$. On the other hand,
\begin{align*}
\sum\deg(S\cup \{v\}) &= \deg(v, A' B'^{k-1}) + \deg(v, V_0 B'^{k-1}) + (k-1) \deg(v, B').
\end{align*}
We thus derive that
\[
\deg(v, A' B'^{k-1}) \ge \binom{|B'|-1}{k-2}|A| -\deg(v, V_0 B'^{k-1}) - (k-1) \deg(v, B').
\]
By Claim~\ref{clm:size} and \eqref{eq:analeft}, it follows that
\begin{align*}
\deg(v, A' B'^{k-1}) &\ge \binom{|B'|-1}{k-2}(|A'|-\e_2|B|) - 2\e_2 |B| \binom{|B'|-1}{k-2} - 2(k-1)\e_1 \binom{|B|}{k-1} \\
&\ge \binom{|B'|-1}{k-2}|A'| - 2k\e_1 \binom{|B|}{k-1}.
\end{align*}
By Part~(1), we now have
\[
\overline{\deg}(v, A_1 B_1^{k-1})\le \overline{\deg}(v, A' B'^{k-1})\le 2k\e_1 \binom{|B|}{k-1}\le 3k\e_1 \binom{|B_1|}{k-1}.
\]

Part (4): By Claim~\ref{fact:typ_l}, for any typical $L\subseteq B'$, we have $\overline\deg(L, A' B'^{k-1})\le 4k\e_1\binom{|B'|-\ell}{k-\ell-1} |A'|$. Thus,
\[
\overline{\deg}(L_0, A_1B_1^{k-1})\le \overline\deg(L_0, A' B'^{k-1})\le 4k\e_1\binom{|B'|-\ell}{k-\ell-1} |A'| \le 5k\e_1\binom{|B_1|}{k-\ell},
\]
where the last inequality holds because $|B'|\le |B_1| +|V(\Q)| \le (1+ \e_1) |B_1|$. The same holds for $L_1$.
\end{proof}

\subsection{Completing the Hamilton cycle}

We finally complete the proof of Theorem \ref{lemE} by applying the following lemma with $X=A_1$, $Y=B_1$, $\rho=5k\e_1$, and $L_0, L_1$.

\begin{lemma}\label{lem:H}
Fix $1\le \ell < k/2$. Let $0<\rho\ll 1$ and $n$ be sufficiently large.
Suppose that $\h$ is a $k$-graph with a partition $V(\h)=X\cup Y$ and the following properties:
\begin{itemize}
\item $|Y| = (2k - 2\ell -1)|X| + \ell$,
\item for every vertex $v\in X$, $\overline \deg(v, Y)\le \rho\binom{|Y|}{k-1}$ and for every vertex $v\in Y$, $\overline \deg (v, XY^{k-1})\le \rho \binom{|Y|}{k-1}$,
\item there are two disjoint $\ell$-sets $L_0, L_1\subset Y$ such that
\begin{equation}\label{eq:L01}
\overline{\deg}(L_0, X Y^{k-1}), \,\overline{\deg}(L_1, X Y^{k-1})\le \rho\binom{|Y|}{k-\ell}.
\end{equation}
\end{itemize}
Then $\h$ contains a Hamilton $\ell$-path with $L_0$ and $L_1$ as ends.
\end{lemma}

In order to prove Lemma \ref{lem:H}, we apply two results of Glebov, Person, and Weps \cite{GPW}. Given $1\le j \le k-1$ and $0\le \rho\le 1$, an ordered set $(x_1,\dots, x_{j})$ is $\rho$-\emph{typical} in a $k$-graph $\G$ if for every $i\in [j]$,
\[
\overline\deg_{\G}( \{x_1, \dots, x_i\} ) \le \rho^{k-i}\binom {|V(\G)| - i}{k-i}.
\]
It was shown in \cite{GPW} that every $k$-graph $\G$ with very large minimum vertex degree contains a tight Hamilton cycle. The proof of \cite[Theorem 2]{GPW} actually shows that we can obtain a tight Hamilton cycle
by extending any fixed tight path of constant length with two typical ends. This implies the following theorem that we will use.

\begin{theorem}\cite{GPW}
\label{thm:GPW}
Given $1\le j\le k$ and $0<\a \ll 1$, there exists an $m_0$ such that the following holds. Suppose that $\G$ is a $k$-graph on $V$ with $|V|= m\ge m_0$ and $\delta_1(\G)\ge (1 - \a)\binom{m-1}{k-1}$. Then given any two disjoint $(22\a)^{\frac1{k-1}}$-typical ordered $j$-sets $(x_1,\dots, x_{j})$ and $(y_1,\dots, y_{j})$, there exists a tight Hamilton path $\mathcal P= x_{j}x_{j-1}\cdots x_1 \cdots \cdots y_1y_2\cdots y_{j}$ in $\G$.
\end{theorem}

We also use \cite[Lemma 3]{GPW}, in which $V^{2k-2}$ denotes the set of all $(2k-2)$-tuples $(v_1, \dots, v_{2k-2})$ such that $v_i\in V$ ($v_i$'s are not necessarily distinct).

\begin{lemma}\cite{GPW}
\label{lem:GPW}
Let $\G$ be the $k$-graph given in Lemma \ref{thm:GPW}. Suppose that $(x_1, \dots, x_{2k-2})$ is selected uniformly at random from $V^{2k-2}$. Then the probability that all $x_i$'s are pairwise distinct and $(x_1, \dots, x_{k-1}), (x_{k}, \dots, x_{2k-2})$ are $(22\a)^{\frac1{k-1}}$-typical is at least $\frac{8}{11}$.
\end{lemma}

\begin{proof}[Proof of Lemma \ref{lem:H}]
In this proof we often write the union $A\cup B\cup \{x\}$ as $A B x$, where $A, B$ are sets and $x$ is an element.

Let $t= |X|$. Our goal is to write $X$ as $\{ x_1, \dots, x_{t} \}$ and partition $Y$ as $\{L_i, R_i, S_i, R_i': i\in [t] \}$ with $|L_i|=\ell$, $|R_i|=|R_i'|=k-2\ell$, and $|S_i|=\ell-1$ such that
\begin{equation}
\label{eq:xiH}
L_{i} R_{i} S_i x_i, \ S_i x_i R_{i}' L_{i+1}\in E(\h)
\end{equation}
for all $i\in [t]$, where $L_{t+1}=L_0$. Consequently
\[
L_1\, R_1\, S_1\, x_1\, R'_1\, L_2\, R_2\, S_2\, x_2\, R'_2\, \cdots \, L_{t}\, R_{t}\, S_t\, x_{t}\, R'_{t}\, L_{t+1}
\]
is the desired Hamilton $\ell$-path of $\h$.

Let $\G$ be the $(k-1)$-graph on $Y$ whose edges are all $(k-1)$-sets $S\subseteq Y$ such that $\deg_{\h}(S, X)> (1-\sqrt{\rho})t$. The following is an outline of our proof. We first find a small subset $Y_0\subset Y$ with a partition $\{L_i, R_i, S_i, R_i': i\in [t_0] \}$ such that for every $x\in X$, we have $L_{i} R_{i} S_i x, S_i x R_{i}' L_{i+1} \in E(\h)$ for many $i\in [t_0]$. Next we apply Theorem~\ref{thm:GPW} to $\G[Y\setminus Y_0]$ and obtain a tight Hamilton path, which, in particular, partitions $Y\setminus Y_0$ into $\{L_i, R_i, S_i, R_i': t_0< i\le t \}$ such that
$L_i R_i S_i, \ S_i R'_i L_{i+1} \in E(\G)$ for $t_0< i\le t$. Finally we apply the Marriage Theorem to find a perfect matching between $X$ and $[t]$ such that \eqref{eq:xiH} holds for all matched $x_i$ and $i$.

We now give details of the proof.  First we claim that
\begin{equation}\label{eq:dG}
\delta_1(\G) \ge (1- 2\sqrt{\rho}) \binom{|Y|-1}{k-2},
\end{equation}
and consequently,
\begin{equation}\label{eq:eG}
\overline{e}(\G)\le 2\sqrt{\rho} \binom{|Y|}{k-1}.
\end{equation}
Suppose instead, some vertex $v\in Y$ satisfies $\overline \deg_\G(v)> 2\sqrt{\rho} \binom{|Y|-1}{k-2}$. Since every non-neighbor $S'$ of $v$ in $\G$ satisfies $\overline \deg_\h(S' v, X)\ge \sqrt{\rho} t$, we have
$\overline \deg_{\h}(v, XY^{k-1}) > 2\sqrt{\rho} \binom{|Y|-1}{k-2} \sqrt{\rho} t$.
Since $|Y| = (2k- 2\ell -1)t + \ell$, we have
\[
\overline \deg_{\h}(v, XY^{k-1}) > 2\rho\frac{|Y|-\ell}{2k-2\ell-1} \binom{|Y|-1}{k-2}
> \rho \frac{|Y|}{k-1} \binom{|Y|-1}{k-2} = \rho \binom{|Y|}{k-1},
\]
contradicting our assumption (the second inequality holds because $|Y|$ is sufficiently large).

Let $Q$ be a $(2k-\ell-1)$-subset of $Y$.  We call $Q$  \emph{good} (otherwise \emph{bad}) if every $(k-1)$-subset of $Q$ is an edge of $\G$ and every $\ell$-set $L\subset Q$ satisfies
\begin{equation}
\label{eq:degGL}
\overline\deg_\G(L) \le \rho^{1/4}\binom{|Y|-\ell}{k-\ell-1}.
\end{equation}
Furthermore, we say $Q$ is \emph{suitable} for a vertex $x\in X$ if $x\cup T\in E(\h)$ for every $(k-1)$-set $T\subset Q$. Note that if a $(2k-\ell-1)$-set is good, by the definition of $\G$, it is suitable for at least $(1- \binom{2k- \ell-1}{k-1}\sqrt{\rho})t$ vertices of $X$. Let $Y'= Y\setminus (L_0\cup L_1)$.

\begin{claim}\label{clm:suitablesets}
For any $x\in X$, at least $(1-\rho^{1/5} )\binom{|Y|}{2k-\ell-1}$ $(2k-\ell-1)$-subsets of $Y'$ are good and suitable for $x$.
\end{claim}

\begin{proof}
Since $\rho + \rho^{1/2} + 3\binom{2k-\ell-1}{\ell} \rho^{1/4}\le \rho^{1/5}$, the claim follows from the following three assertions:
\begin{itemize}
\item At most $2\ell \binom{|Y|-1}{2k - \ell -2}\le \rho \binom{|Y|}{2k- \ell -1}$ $(2k-\ell-1)$-subsets of $Y$ are not subsets of $Y'$.
\item Given $x\in X$, at most $\rho^{1/2} \binom{|Y|}{2k-\ell-1}$ $(2k-\ell-1)$-sets in $Y$ are not suitable for $x$.
\item At most $3\binom{2k-\ell-1}{\ell} \rho^{1/4} \binom{|Y|}{2k-\ell-1}$ $(2k-\ell-1)$-sets in $Y$ are bad.
\end{itemize}

The first assertion holds because $|Y\setminus Y'|= 2\ell$. The second assertion follows from the degree condition of $\h$, namely, for any $x\in X$, the number of $(2k-\ell-1)$-sets in $Y$ that are not suitable for $x$ is at most ${\rho}\binom{|Y|}{k-1}\binom{|Y|-k+1}{k-\ell}\le \sqrt{\rho}\binom{|Y|}{2k-\ell-1}$.

To see the third one, let $m$ be the number of $\ell$-sets $L\subseteq Y$ that fail \eqref{eq:degGL}.
By \eqref{eq:eG},
\[
m \frac{\rho^{1/4}\binom{|Y|-\ell}{k-\ell-1} }{ \binom{k-1}{\ell} }\le  \overline{e}(\G)\le 2\sqrt{\rho} \binom{|Y|}{k-1},
\]
which implies that $m\le 2\rho^{1/4} \binom{|Y|}{\ell}$.
Thus at most
\[
2\rho^{1/4}\binom{|Y|}{\ell}\cdot \binom{|Y| - \ell}{2k-2\ell-1}
\]
$(2k - \ell -1)$-subsets of $Y$ contain an $\ell$-set $L$ that fails \eqref{eq:degGL}.
On the other hand, by \eqref{eq:eG}, at most
\[
\overline{e}(\G) \binom{ |Y|-k+1 }{k- \ell} \le 2\sqrt{\rho} \binom{|Y|}{k-1} \binom{ |Y|-k+1 }{k- \ell}
\]
$(2k - \ell -1)$-subsets of $Y$ contain a non-edge of $\G$. Putting these together, the number of bad $(2k-\ell-1)$-sets in $Y$ is at most
\[
2\rho^{1/4}\binom{|Y|}{\ell} \binom{|Y|-\ell}{2k-2\ell-1} + 2\sqrt{\rho} \binom{|Y|}{k-1} \binom{ |Y|-k+1 }{k- \ell}  \le 3\binom{2k-\ell-1}{\ell}\rho^{1/4} \binom{|Y|}{2k-\ell-1},
\]
as $\rho\ll 1$.
\end{proof}

Let $\F_0$ be the set of good $(2k-\ell-1)$-sets in $Y'$.
We will pick a family of disjoint good $(2k-\ell-1)$-sets in $Y'$ such that for any $x\in X$, many members of this family are suitable for $x$. To achieve this, we pick a family $\mathcal F$ by selecting each member of $\F_0$ randomly and independently with probability $p=6\sqrt{\rho} |Y|/\binom{|Y|}{2k-\ell-1}$.
Then $|\F|$ follows the binomial distribution $B(|\F_0|, p)$ with expectation $\mathbb{E}(|\F|) = p |\F_0| \le p\binom{|Y|}{2k - \ell - 1}$.
Furthermore, for every $x\in X$, let $f(x)$ denote the number of members of $\F$ that are suitable for $x$. Then $f(x)$ follows the binomial distribution $B(N, p)$ with $N\ge (1-\rho^{1/5} )\binom{|Y|}{2k-\ell-1}$ by Claim \ref{clm:suitablesets}. Hence $\mathbb{E}(f(x))\ge p(1-\rho^{1/5} )\binom{|Y|}{2k-\ell-1}$. 
Since there are at most $\binom{|Y|}{2k-\ell-1}\cdot (2k-\ell-1)\cdot \binom{|Y|-1}{2k-\ell-2}$ pairs of intersecting $(2k-\ell-1)$-sets in $Y$,
the expected number of intersecting pairs of $(2k-\ell-1)$-sets in $\mathcal F$ is at most
\[
p^2 \binom{|Y|}{2k-\ell-1}\cdot (2k-\ell-1)\cdot \binom{|Y|-1}{2k-\ell-2}=36(2k-\ell-1)^2 {\rho} |Y|.
\]

By Chernoff's bound (the first two properties) and Markov's bound (the last one), we can find a family $\mathcal F$ of good $(2k-\ell-1)$-subsets of $Y'$ that satisfies
\begin{itemize}
\item
$|\mathcal F|\le 2p \binom{|Y'|}{2k-\ell-1} \le 12\sqrt{\rho} |Y|$,
\item
for any vertex $x\in X$, at least
$\frac{p}2 (1- \rho^{1/5})\binom{|Y|}{2k-\ell-1}  \ge 2\sqrt{\rho}|Y|$
members of $\mathcal F$ are suitable for $x$.
\item
the number of intersecting pairs of $(2k-\ell-1)$-sets in $\mathcal F$ is at most $72(2k-\ell-1)^2 {\rho} |Y|$.
\end{itemize}
After deleting one $(2k-\ell-1)$-set from each of the intersecting pairs from $\F$, we obtain a family $\mathcal F'\subseteq \F$ consisting of at most $12\sqrt{\rho} |Y|$ disjoint good $(2k-\ell-1)$-subsets of $Y'$ and for each $x\in X$, at least
\begin{equation}\label{eq:F'}
2\sqrt{\rho}|Y| - 72(2k-\ell-1)^2 {\rho} |Y|\ge \frac32\sqrt{\rho}|Y|
\end{equation}
members of $\mathcal F'$ are suitable for $x$.

Denote $\F'$ by $\{Q_2, Q_4, \dots, Q_{2q} \}$ for  some $q\le 12\sqrt{\rho} |Y|$. We arbitrarily partition each $Q_{2i}$ into $L_{2i} \cup P_{2i} \cup L_{2i+1}$ such that $|L_{2i}| = |L_{2i+1}|= \ell$ and $|P_{2i}| =2k - 3\ell -1$.
Since $Q_{2i}$ is good, both $L_{2i}$ and $L_{2i+1}$ satisfy \eqref{eq:degGL}. We claim that $L_0$ and $L_1$ satisfy \eqref{eq:degGL} as well. Let us show this for $L_0$. By the definition of $\G$, the number of $XY^{k-\ell-1}$-sets $T$ such that $T\cup L_0\not\in E(\h)$ is at least $\overline\deg_\G(L_0) \sqrt{\rho}t$. Using \eqref{eq:L01}, we derive that $
 \overline\deg_\G(L_0) \sqrt{\rho}t \le \rho \binom{|Y|}{k- \ell}$. Since $|Y|\le (2k- 2\ell)t$, it follows that
$\overline\deg_\G(L_0)\le 2\sqrt{\rho} \binom{|Y| -1}{k- \ell - 1}\le  {\rho}^{1/4} \binom{|Y| -\ell}{k- \ell - 1}$.

Next we greedily find disjoint $(2k- 3\ell -1)$-sets $P_1, P_3, \dots, P_{2q-1}$ from $Y'\setminus \bigcup_{i=1}^q Q_{2i}$ such that for each $i\in [q]$, every $(k- \ell -1)$-subset of $P_{2i-1}$ is a common neighbor of $L_{2i-1}$ and $L_{2i}$ in $\G$. 
Suppose that we have found $P_1, P_3, \dots, P_{2i-1}$ for some $i<q$. 
Since both $L_{2i-1}$ and $L_{2i}$ satisfy \eqref{eq:degGL}, at most
\[
2 \cdot {\rho}^{1/4}\binom{|Y|-\ell}{k-\ell-1} \binom{|Y| - k +1}{k- 2\ell}
\]
$(2k- 3\ell -1)$-subsets of $Y$ contain a non-neighbor of $L_{2i-1}$ or $L_{2i}$. 
Thus, the number of $(2k- 3\ell -1)$-sets that can be chosen as $P_{2i+1}$ is at least
\[
\binom{|Y'| - (2k- 2\ell-1)2q}{2k - 3\ell - 1} - 2 \cdot {\rho}^{1/4}\binom{|Y|-\ell}{k-\ell-1} \binom{|Y| - k +1}{k- 2\ell} > 0,
\]
as $q\le 12\sqrt{\rho} |Y|$ and $\rho\ll 1$.

Let $Y_1= Y' \setminus  \bigcup_{i=1}^q ( P_{2i-1} \cup Q_{2i}) $ and $\G'=\G[Y_1]$. Then $|Y_1|=|Y'|-(2k-2\ell-1)2q$. Since $\overline \deg_{\G'}(v)\le \overline \deg_\G(v)$ for every $v\in Y_1$, we have, by \eqref{eq:dG},
\[
\delta_1(\G') \ge \binom{|Y_1|-1}{k-2} -  2\sqrt{\rho} \binom{|Y|-1}{k-2}\ge (1-3\sqrt{\rho})\binom{|Y_1|-1}{k-2}.
\]

Let  $\a=3\sqrt{\rho}$ and $\rho_0=(22\a)^{\frac1{k-1}}$.
We want to find two disjoint $\rho_0$-typical ordered $(k-\ell-1)$-subsets $(x_1,\dots, x_{k-\ell-1})$ and $(y_1,\dots, y_{k-\ell-1})$ of $Y_1$ such that
\begin{equation}\label{eq:endG}
L_{2q+1} \cup \{x_1,\dots, x_{k-\ell-1}\}, \ L_0 \cup \{y_1,\dots, y_{k-\ell-1}\} \in E(\G).
\end{equation}
To achieve this, we choose $(x_1,\dots, x_{k-1}, y_1,\dots, y_{k-1})$ from ${Y_1}^{2k-2}$ uniformly at random. By Lemma \ref{lem:GPW}, with probability at least $\frac8{11}$,  $(x_1,\dots, x_{k-\ell-1})$ and $(y_1,\dots, y_{k-\ell-1})$ are two disjoint
ordered $\rho_0$-typical $(k-\ell-1)$-sets. Since $L_0$ satisfies \eqref{eq:degGL}, at most $(k-\ell -1)! \rho^{1/4} \binom{|Y|-\ell}{k-\ell-1}$ ordered $(k-\ell-1)$-subsets of $Y$ are not neighbors of $L_0$ (the same holds for $L_{2q+1}$). Thus  \eqref{eq:endG} fails with probability at most $2(k-\ell -1)! \rho^{1/4}$, provided that $x_1,\dots, x_{k-\ell-1}$, $y_1,\dots, y_{k-\ell-1}$ are all distinct. Therefore the desired $(x_1,\dots, x_{k-\ell-1})$ and $(y_1,\dots, y_{k-\ell-1})$ exist.

Next we apply Theorem \ref{thm:GPW} to $\G'$ and obtain a tight Hamilton path
\[
\mathcal{P}= x_{k-\ell-1}x_{k-\ell-2}\cdots x_1 \cdots \cdots y_1y_2\cdots y_{k-\ell-1}.
\]
Following the order of $\mathcal{P}$, we partition $Y_1$ into
\[
R_{2q+1}, S_{2q+1}, R_{2q+1}', L_{2q+2},\dots, L_t, R_{t}, S_{t}, R_{t}'
\]
such that $|L_i|=\ell$, $|R_i|=|R_i'|=k-2\ell$, and $|S_i|=\ell-1$. Since $\mathcal{P}$ is a tight path in $\G$, we have
\begin{equation}\label{eq:goodQ}
L_i R_i S_i, \ S_i R'_i L_{i+1} \in E(\G)
\end{equation}
for $2q+2\le i\le t-1$. Letting $L_{t+1}= L_0$, by \eqref{eq:endG}, we also have \eqref{eq:goodQ} for $i=2q+1$ and $i=t$.

We now arbitrarily partition $P_i$, $1\le i\le 2q$ into $R_i \cup S_i\cup R'_i$ such that $|R_i|=|R_i'|=k-2\ell$, and $|S_i|=\ell-1$. By the choice of $P_i$,  \eqref{eq:goodQ} holds for $1\le i \le 2q$.

Consider the bipartite graph $\Gamma$ between $X$ and $Z := \{z_1, z_2, \dots, z_{t}\}$ such that $x\in X$ and $z_i\in Z$ are adjacent if and only if $L_i R_i S_i x, x S_i R'_i L_{i+1}\in E(\h)$. For every $i\in [t]$, since \eqref{eq:goodQ} holds, we have $\deg_{\Gamma} (z_i) \ge (1 - 2\sqrt \rho)t$ by the definition of $\G$. Let $Z'= \{z_{2q+1}, \dots, z_{t}\}$ and $X_0$ be the set of $x\in X$ such that $\deg_{\Gamma}(x, Z')\le |Z'|/2$. Then
\[
|X_0| \frac{|Z'|}2 \le \sum_{x\in X} \overline{\deg}_{\Gamma}(x, Z') \le 2 \sqrt \rho t\cdot |Z'|,
\]
which implies that
$|X_0| \le 4 \sqrt\rho t = 4 \sqrt{\rho} \frac{|Y|-\ell}{2k-2\ell-1}\le \frac{4}{3} \sqrt\rho |Y|$ (note that $2k-2\ell-1\ge k\ge 3$).

We now find a perfect matching between $X$ and $Z$ as follows.
\begin{enumerate}
\item[Step 1:]
Each $x\in X_0$ is matched to some $z_{2i}$, $i\in [q]$ such that the corresponding $Q_{2i}\in \F'$ is suitable for $x$ (thus $x$ and $z_{2i}$ are adjacent in $\Gamma$) -- this is possible because of \eqref{eq:F'} and
$|X_0| \le\frac43 \sqrt\rho |Y|$.
\item[Step 2:]
Each of the unused $z_{i}$, $i\in [2q]$ is matched to a vertex in $X\setminus X_0$ -- this is possible because $\deg_{\Gamma} (z_i) \ge (1 - 2\sqrt \rho)t \ge |X_0| + 2q$.
\item[Step 3:]
Let $X'$ be the set of the remaining vertices in $X$. Then $|X'|= t - 2q = |Z'|$.
Now consider the induced subgraph $\Gamma'$ of $\Gamma$ on $X' \cup Z'$. Since $\delta(\Gamma')\ge |X'|/2$, the Marriage Theorem provides a perfect matching in $\Gamma'$.
\end{enumerate}
The perfect matching between $X$ and $Z$ gives rise to the desired Hamilton path of $\h$.
\end{proof}

\section{Concluding Remarks}

Let $h^{\ell}_d(k, n)$ denote the minimum integer $m$ such that every $k$-graph $\h$ on $n$ vertices with minimum $d$-degree $\delta_d(\h) \ge m$ contains a Hamilton $\ell$-cycle (provided that $k- \ell$ divides $n$). In this paper we determined $h^{\ell}_{k-1}(k, n)$ for all $\ell<k/2$ and sufficiently large $n$. Unfortunately our proof does not give $h^{\ell}_{k-1}(k, n)$ for all $k, \ell$ such that $k-\ell$ does not divide $k$ even though we believe that $h^{\ell}_{k-1}(k, n)= \frac{n}{\lceil \frac{k}{k-\ell} \rceil (k-\ell)}$. In fact, when $k-\ell$ does not divide $k$, if we can prove a path-cover lemma similar to Lemma~\ref{lemP}, then we can follow the proof in \cite{KMO} to solve the nonextremal case. When $\ell \ge k/2$, we cannot define $\Y_{k,2\ell}$ so the current proof of Lemma~\ref{lemP} fails. In addition, when $\ell \ge k/2$, the extremal case becomes complicated as well.

The situation is quite different when $k-\ell$ divides $k$.
When $k$ divides $n$, one can easily construct a $k$-graph $\h$ such that
$\delta_{k-1}(\h)\ge \frac{n}{2} - k$ and yet $\h$ contains no perfect matching and consequently no Hamilton $\ell$-cycle for any $\ell$ such that $k-\ell$ divides $k$. A construction in \cite{MaRu} actually shows that $h_{k-1}^{\ell}(k,n)\ge \frac n2-k$ whenever $k-\ell$ divides $k$, even when $k$ does not divide $n$.
The exact value of $h^{\ell}_d(k, n)$, 
when $k-\ell$ divides $k$, is not known except for $h_2^2(3,n)= \lfloor n/2 \rfloor$ given in \cite{RRS11}.
In the forthcoming paper \cite{HZk2}, we determine $h^{k/2}_d(k, n)$ exactly for even $k$ and any $d\ge k/2$.

Let $t_d(n, F)$ denote the minimum integer $m$ such that every $k$-graph $\h$ on $n$ vertices with minimum $d$-degree $\delta_d(\h) \ge m$ contains a perfect $F$-tiling.
One of the first results on hypergraph tiling was $t_2(n, \Y_{3,2})= n/4 + o(n)$ given by K\"uhn and Osthus \cite{KO}. The exact value of $t_2(n, \Y_{3,2})$ was determined recently by Czygrinow, DeBiasio, and Nagle \cite{CDN}. We \cite{HZ3} determined $t_1(n, \Y_{3,2})$ very recently.
The key lemma in our proof, Lemma~\ref{lem:F}, shows that every $k$-graph $\h$ on $n$ vertices with $\delta_{k-1}(\h)\ge  (\frac{1}{2k-b} - o(1)) n$ either contains an almost perfect $\Y_{k,b}$-tiling or is in the extremal case. Naturally this raises a question: what is $t_{k-1}(n, \Y_{k,b})$?
Mycroft \cite{My14} recently proved a general result on tiling $k$-partite $k$-graphs, which implies that $t_{k-1}(n, \Y_{k,b})= \frac{n}{2k-b} + o(n)$.
The lower bound comes from the following construction. Let $\h_0$ be the $k$-graph on $n\in (2k-b)\mathbb N$ vertices such that $V(\h_0)=A\cup B$ with $|A|=\frac{n}{2k-b} - 1$, and $E(\h_0)$ consists of all $k$-sets intersecting $A$ and some $k$-subsets of $B$ such that $\h_0[B]$ contains no copy of $\Y_{k,b}$.
Thus, $\delta_{k-1}(\h_0)\ge \frac{n}{2k-b}-1$. Since every copy of $\Y_{k,b}$ contains at least one vertex in $A$, there is no perfect $\Y_{k,b}$-tiling in $\h_0$.
We believe that one can find a matching upper bound by the absorbing method (similar to the proof in \cite{CDN}). In fact, since we already proved Lemma~\ref{lem:F}, it suffices to prove an absorbing lemma and the extremal case.

\section*{Acknowledgement}
We thank two referees for their valuable comments that improved the presentation of this paper.

\bibliographystyle{plain}
\bibliography{Jan2014}

\begin{thebibliography}{10}

\bibitem{BHS}
E.~Bu{\ss}, H.~H{\`a}n, and M.~Schacht.
\newblock Minimum vertex degree conditions for loose {H}amilton cycles in
  3-uniform hypergraphs.
\newblock {\em J. Combin. Theory Ser. B}, 103(6):658--678, 2013.

\bibitem{CDN}
A.~Czygrinow, L.~DeBiasio, and B.~Nagle.
\newblock Tiling 3-uniform hypergraphs with ${K}_4^3-2e$.
\newblock {\em Journal of Graph Theory}, 75(2):124--136, 2014.

\bibitem{CzMo}
A.~Czygrinow and T.~Molla.
\newblock Tight codegree condition for the existence of loose {H}amilton cycles
  in 3-graphs.
\newblock {\em SIAM J. Discrete Math.}, 28(1):67--76, 2014.

\bibitem{Dirac}
G.~A. Dirac.
\newblock Some theorems on abstract graphs.
\newblock {\em Proc. London Math. Soc. (3)}, 2:69--81, 1952.

\bibitem{EKR}
P.~Erd{\H{o}}s, C.~Ko, and R.~Rado.
\newblock Intersection theorems for systems of finite sets.
\newblock {\em Quart. J. Math. Oxford Ser. (2)}, 12:313--320, 1961.

\bibitem{GPW}
R.~Glebov, Y.~Person, and W.~Weps.
\newblock On extremal hypergraphs for {H}amiltonian cycles.
\newblock {\em European J. Combin.}, 33(4):544--555, 2012.

\bibitem{HS}
H.~H\`an and M.~Schacht.
\newblock Dirac-type results for loose {Hamilton} cycles in uniform
  hypergraphs.
\newblock {\em Journal of Combinatorial Theory. Series B}, 100:332--346, 2010.

\bibitem{HZk2}
J.~Han and Y.~Zhao.
\newblock Minimum degree conditions for {Hamilton} $(k/2)$-cycles in
  $k$-uniform hypergraphs.
\newblock {\em manuscript}.

\bibitem{HZ1}
J.~Han and Y.~Zhao.
\newblock Minimum degree thresholds for loose {Hamilton} cycle in 3-graphs.
\newblock {\em submitted}.

\bibitem{HZ3}
J.~Han and Y.~Zhao.
\newblock Minimum degree thresholds for ${C}_4^3$-tiling.
\newblock {\em Journal of Graph Theory, in press}, DOI: 10.1002/jgt.21833.

\bibitem{KK}
G.~Katona and H.~Kierstead.
\newblock Hamiltonian chains in hypergraphs.
\newblock {\em Journal of Graph Theory}, 30(2):205--212, 1999.

\bibitem{KKMO}
P.~Keevash, D.~K\"uhn, R.~Mycroft, and D.~Osthus.
\newblock Loose {Hamilton} cycles in hypergraphs.
\newblock {\em Discrete Mathematics}, 311(7):544--559, 2011.

\bibitem{KMO}
D.~K\"uhn, R.~Mycroft, and D.~Osthus.
\newblock Hamilton $\ell$-cycles in uniform hypergraphs.
\newblock {\em Journal of Combinatorial Theory. Series A}, 117(7):910--927,
  2010.

\bibitem{KO}
D.~K\"uhn and D.~Osthus.
\newblock Loose {Hamilton} cycles in 3-uniform hypergraphs of high minimum
  degree.
\newblock {\em Journal of Combinatorial Theory. Series B}, 96(6):767--821,
  2006.

\bibitem{KuOs14ICM}
D.~K{\"u}hn and D.~Osthus.
\newblock Hamilton cycles in graphs and hypergraphs: an extremal perspective.
\newblock {\em Proceedings of the International Congress of Mathematicians
  2014, Seoul, Korea}, Vol 4:381--406, 2014.

\bibitem{MaRu}
K.~Markstr\"{o}m and A.~Ruci\'{n}ski.
\newblock Perfect {M}atchings (and {H}amilton {C}ycles) in {H}ypergraphs with
  {L}arge {D}egrees.
\newblock {\em Eur. J. Comb.}, 32(5):677--687, July 2011.

\bibitem{My14}
R.~Mycroft.
\newblock Packing k-partite k-uniform hypergraphs.
\newblock {\em submitted}.

\bibitem{RR}
V.~R\"odl and A.~Ruci\'nski.
\newblock Dirac-type questions for hypergraphs — a survey (or more problems
  for endre to solve).
\newblock {\em An Irregular Mind}, Bolyai Soc. Math. Studies 21:561--590, 2010.

\bibitem{RRS06}
V.~R\"odl, A.~Ruci\'nski, and E.~Szemer\'edi.
\newblock A {D}irac-type theorem for 3-uniform hypergraphs.
\newblock {\em Combinatorics, Probability and Computing}, 15(1-2):229--251,
  2006.

\bibitem{RRS08}
V.~R\"odl, A.~Ruci\'nski, and E.~Szemer\'edi.
\newblock An approximate {D}irac-type theorem for k-uniform hypergraphs.
\newblock {\em Combinatorica}, 28(2):229--260, 2008.

\bibitem{RRS11}
V.~R\"odl, A.~Ruci\'nski, and E.~Szemer\'edi.
\newblock Dirac-type conditions for {Hamiltonian} paths and cycles in 3-uniform
  hypergraphs.
\newblock {\em Advances in Mathematics}, 227(3):1225--1299, 2011.

\bibitem{Sze}
E.~Szemer{\'e}di.
\newblock Regular partitions of graphs.
\newblock In {\em Probl\`emes combinatoires et th\'eorie des graphes ({C}olloq.
  {I}nternat. {CNRS}, {U}niv. {O}rsay, {O}rsay, 1976)}, volume 260 of {\em
  Colloq. Internat. CNRS}, pages 399--401. CNRS, Paris, 1978.

\end{thebibliography}

\end{document}